\documentclass[preprint]{amsart}

\usepackage{amsmath, amssymb}
\usepackage{color}
\usepackage{latexsym}
\usepackage{indentfirst}
\usepackage{graphicx}
\usepackage{placeins}
\usepackage{booktabs}
\usepackage{algorithm}
\usepackage{algorithmic}
\usepackage{multirow}
\usepackage{subfig}
\theoremstyle{plain}
\newtheorem{theorem}{Theorem}

\newtheorem{prop}[theorem]{Proposition}

\theoremstyle{remark}
\newtheorem*{remark}{Remark}

\DeclareMathOperator{\Tr}{Tr}

\newcommand{\ud}{\,\mathrm{d}}

\newcommand{\Or}{\mathcal{O}}

\newcommand{\wt}[1]{\widetilde{#1}}

\newcommand{\abs}[1]{\left\lvert#1\right\rvert}
\newcommand{\norm}[1]{\left\lVert#1\right\rVert}

\renewcommand{\Re}{\mathrm{Re}} 

\newcommand{\I}{\imath}

\newcommand{\barint}{\kern4pt \raise3.4pt\hbox{\vrule height.6pt
    width7pt} \kern-11pt \int}

\newcommand{\REV}[1]{#1}

\title[RANDOMIZED ESTIMATION OF SPECTRAL DENSITIES MADE
ACCURATE]{Randomized estimation of spectral densities of large matrices made accurate} 
\author[L. LIN]{Lin Lin}
\address[Lin Lin]{Department of Mathematics, University of California,
Berkeley, Berkeley, CA 94720.\\
Computational Research Division, Lawrence Berkeley National
Laboratory, Berkeley, CA 94720.}
\email{linlin@math.berkeley.edu}

\begin{document}

\maketitle

\begin{abstract}
For a large Hermitian matrix $A\in \mathbb{C}^{N\times N}$, it is often
the case that the only affordable operation is matrix-vector
multiplication. In such case, randomized method is a powerful way to
estimate the spectral density (or density of states) of $A$.  However,
randomized methods developed so far for estimating spectral densities
only extract information from different random vectors independently,
and the accuracy is therefore inherently limited to
$\Or(1/\sqrt{N_{v}})$ where $N_{v}$ is the number of random vectors.  In
this paper we demonstrate that the ``$\Or(1/\sqrt{N_{v}})$ barrier'' can
be overcome by taking advantage of the correlated information of random
vectors when properly filtered by polynomials of $A$. Our method uses
the fact that the estimation of the spectral density essentially
requires the computation of the trace of a series of \REV{matrix
functions that are numerically low rank}. By repeatedly applying $A$ to
the same set of random vectors and taking different linear combination
of the results, we can sweep through the entire spectrum of $A$ by
building such low rank decomposition at different parts of the spectrum.
Under some assumptions, we demonstrate that a robust and efficient
implementation of such spectrum sweeping method can compute the spectral
density accurately with $\Or(N^2)$ computational cost and $\Or(N)$
memory cost.  Numerical results indicate that the new method can
significantly outperform existing randomized methods in terms of
accuracy.  
As an application, we demonstrate a way to
accurately compute a trace of a smooth matrix function, by carefully
balancing the smoothness of the integrand and the regularized density of
states using a deconvolution procedure.
\end{abstract}








\section{Introduction}

Given an $N\times N$ Hermitian matrix $A$, the \textit{spectral
density}, also commonly referred to as the \emph{density of states} (DOS), 
is formally defined as
\begin{equation}
  \phi(t) = \frac{1}{N}  \sum_{i=1}^N \delta(t - \lambda_i).
  \label{eq:DOS0}
\end{equation}
Here $\delta$ is the Dirac distribution commonly referred to as the
Dirac
$\delta$-``function'' (see
e.g.~\cite{Schwartz1966,ByronFuller1992,RichtmyerBeiglbock1981}), and
the $\lambda_{i}$'s are the eigenvalues of $A$, assumed here to be
labeled non-decreasingly.  

The DOS is an important quantity in many physics problems, in particular
in quantum physics,
and a large volume of numerical methods were developed by physicists and
chemists~\cite{Ducastelle1970,Turek88,DraboldSankey1993,WheelerBlumstein1972,SilverRoder1994,Wang1994,WeiseWelleinAlvermannEtAl2006,CovaciPeetersBerciu2010,JungCzychollKettemann2012,SeiserPettiforDrautz2013,HaydockHeineKelly1972,ParkerZhuHuangEtAl1996}
for this purpose. Besides being used as a \REV{qualitative} visualization tool for understanding
spectral characteristics of the matrix, the DOS can also be used as to
quantitatively compute the trace of a matrix function, as given in the
formal formulation below
\begin{equation}
  \Tr[f(A)] = \sum_{i=1}^{N} f(\lambda_{i}) \equiv
  N \int_{-\infty}^{\infty}
  f(t)\phi(t) \ud t.
  \label{eqn:tracefA}
\end{equation}
Here $f(t)$ is a smooth function, and the formal integral in
Eq.~\eqref{eqn:tracefA} should be interpreted in the sense of
distribution. 

If one had access to all the eigenvalues of $A$, the task of computing
the DOS would become a trivial one.  However, in many applications, the
dimension of $A$ is large. The computation of its entire spectrum is
prohibitively expensive, and a procedure that relies entirely on
multiplications of $A$ with vectors is the only viable approach.
Fortunately, in many applications $A$ only has $\Or(N)$ nonzero entries,
and therefore the cost of matrix-vector multiplication, denoted by
$c_{\mathrm{matvec}}$, is $\Or(N)$. In some other cases the matrix
is a dense matrix but fast matrix-vector multiplication method still
exists with $\Or(N \log^{p} N)$ cost, where $p$ is a integer that is not
too large. This is the case when
the matrix-vector multiplication can be carried out effectively with
fast algorithms, such as the fast Fourier transform (FFT), the fast
multipole method (FMM)~\cite{GreengardRokhlin1987}, the hierarchical
matrix~\cite{Hackbusch1999}, and the fast butterfly
algorithm~\cite{CandesDemanetYing2009}, to name a few.

Rigorously speaking, the DOS is a distribution and cannot be directly
approximated by smooth functions.  In order to assess the accuracy of a
given numerical scheme for estimating the DOS, the DOS must be properly
\textit{regularized}.  The basic idea for estimating the DOS is to
first expand the regularized DOS using simple functions
such as polynomials.  Then it can be shown that the estimation of the
DOS can be obtained by computing the trace of a polynomial of
$A$, which can then be estimated by 
repeatedly applying $A$ to a set of random vectors.  This procedure has been discovered
more or less independently by statisticians~\cite{Hutchinson1989} and by
physicists and chemists~\cite{SilverRoder1994,Wang1994}, and will be
referred to as Hutchinson's method in the following. In physics such
method is often referred to as the kernel polynomial method
(KPM)~\cite{WeiseWelleinAlvermannEtAl2006} with a few different 
variants.  A recent review on the choice of regularization and different
numerical methods for estimating the DOS is given in~\cite{LinSaadYang2015}. There are also a variety of randomized
estimators that can be used in Hutchinson's method, and the quality of
different estimators is analyzed in~\cite{AvronToledo2011}.

\textbf{Contribution.} To the extent of our
knowledge, all randomized methods so far for estimating the DOS are
based on different variants of Hutchinson's method.
These methods estimate the DOS by averaging 
the information obtained from $N_{v}$ random vectors directly. 
The numerical error, when properly defined, decays asymptotically as
$\Or(1/\sqrt{N_v})$. As a result, high accuracy is difficult to achieve:
every extra digit of accuracy requires increasing the number of random
vectors by $100$ fold.

In this work, we demonstrate that the accuracy for estimating the
regularized DOS can be significantly improved by making use of the
\textit{correlated information} obtained among different random vectors.
We use the fact that each point of the DOS can be evaluated as the trace
of a \REV{numerically} low rank matrix, and such trace can be evaluated by repeatedly
applying $A$ to a small number of random vectors, and by taking
certain linear combination of the resulting vectors.  If different set of
random vectors were needed for different points on the spectrum the
method will be prohibitively expensive.  However, we demonstrate that it
is possible to use \textit{the same set of random vectors} to ``sweep
through'' in principle the entire spectrum.  Therefore we call our
method a ``spectrum sweeping method''.


\REV{Our numerical results indicate that the spectrum sweeping method
can significantly outperform Hutchinson type methods in terms of
accuracy,  as the number of random vectors $N_{v}$ becomes large.}
However, the computational cost and the storage cost can still be large
when the DOS needs to be evaluated at a large number of points.
Furthermore, the accuracy of the spectrum sweeping method may be
compromised when the right number of randomized vectors is not known
\textit{a priori}.  We develop a robust and efficient implementation of
the spectrum sweeping method to overcome these two problems.  Under
certain assumption on the distribution of eigenvalues of the matrix $A$,
and the cost of the matrix-vector multiplication is $\Or(N)$, we
demonstrate that the computational cost of the new method scales as
$\Or(N^2)$ and the storage cost scales as $\Or(N)$ for increasingly
large matrix dimension $N$. We also demonstrate that the new method for
evaluating the DOS can be useful for accurate trace estimation as in
Eq.~\eqref{eqn:tracefA}.

\textbf{Other related works.} The spectrum sweeping method is not to be
confused with another set of methods under the name of ``spectrum
slicing''
methods~\cite{Parlett1980,SakuraiSugiura2003,Polizzi2009,SchofieldChelikowskySaad2012,FangSaad2012,AktulgaLinHaineEtAl2014}.
The idea of the spectrum slicing methods is still to obtain a partial
diagonalization of the
matrix $A$. The main advantage of spectrum slicing methods is enhanced
parallelism compared to conventional diagonalization methods.  Due to the natural orthogonality of eigenvectors
corresponding to distinct eigenvalues of a Hermitian matrix, the
computational cost for each set of processors handling different parts
of the spectrum can be reduced compared to direct diagonalization
methods.  However, the overall scaling for spectrum slicing methods is
still $\Or(N^3)$ when a large number of eigenvalues and eigenvectors are
to be computed. 

%
%

\textbf{Notation.} In linear algebra notation, a vector
$w\in \mathbb{C}^{m}$ is always treated as a column vector, and its conjugate transpose
is denoted by $w^{*}$. For a randomized matrix $B\in \mathbb{C}^{m\times
n}$, its entry-wise expectation value is denoted by $\mathbb{E} [B]$ and
its entry-wise variance is denoted by $\mathrm{Var} [B]$.  We call $B\in
\mathbb{C}^{m\times n}$ a (real) random Gaussian matrix, if each entry
of $B$ is real, and follows independently the normal distribution
$\mathcal{N}(0,1)$.  In the case when $n=1$, $B$ is called a (real)
random Gaussian vector. \REV{The imaginary unit is denoted by $\I$.}

The paper is organized as follows. In section~\ref{sec:DOS} we introduce
the DOS estimation problem. We also demonstrate the
Delta-Gauss-Chebyshev (DGC) method, which is a variant of the kernel
polynomial method, to estimate the DOS.  We develop in
section~\ref{sec:doslowrank} the spectrum sweeping method based on the
randomized estimation of the trace of \REV{numerically} low rank matrices, and demonstrate
a robust and efficient implementation of the spectrum sweeping method in
section~\ref{sec:robustdoslowrank}.  We show how the DOS
estimation method can be used to effective compute the trace of a matrix
function in section~\ref{sec:traceest}.  Following the numerical results
in section~\ref{sec:numer}, we conclude and discuss the future work in
section~\ref{sec:conclusion}.

\section{Density of state estimation for large matrices}
\label{sec:DOS}

Without loss of generality, we shall assume that the spectrum of
$A$ is contained in the interval $(-1,1)$. \REV{For a general matrix with spectrum
contained in the interval $(a,b)$, we can apply a spectral transformation
\[
\wt{A} = \frac{2A - (a+b) I}{b-a}.
\]
The spectrum of $\wt{A}$ is contained in $(-1,1)$, and all the
discussion below can be applied to $\wt{A}$. }
The fact that the spectral density $\phi(t)$ is defined in terms of
Dirac $\delta$-functions suggests that no smooth function can converge
to the spectral density in the limit of high
resolution, \REV{in the usual $L^{p}$ norms ($p\ge 1$)} ~\cite{LinSaadYang2015}.  In order to compare different
numerical approximations to the spectral density in a meaningful way,
the DOS should be regularized.  One simple method is to employ
a Gaussian regularization
\begin{equation}
  \phi_{\sigma}(t) = \sum_{i=1}^{N} g_{\sigma}(t-\lambda_{i}) =
  \Tr[g_{\sigma}(tI-A)].
  \label{}
\end{equation}
Here 
\begin{equation}
  g_{\sigma}(s) = \frac{1}{N \sqrt{2\pi \sigma^2}}
  e^{-\frac{s^2}{2\sigma^2}}
  \label{eqn:gaussian}
\end{equation}
is a Gaussian function.  In the following our goal is to compute the
smeared DOS $\phi_{\sigma}$. 

The key of computing the DOS is the estimation 
of the trace of a matrix function without diagonalizing the matrix.
To the extent of our knowledge, randomized methods developed so far 
are based on Hutchinson's method or its
variants~\cite{Hutchinson1989,AvronToledo2011}. \REV{The following
simple and yet useful theorem is a simple variant of Hutchinson's method
and explains how the method works.}
\begin{theorem}[\REV{Hutchinson's method}]\label{thm:EVarTrace}
  Let $A\in \mathbb{C}^{N\times N}$ be a Hermitian matrix, and
  \REV{$w\in\mathbb{R}^{N}$} be a random Gaussian vector, 
  then
  \begin{equation}
    \mathbb{E} [w^{*} A w] = \Tr[A], \quad \mathrm{Var} [w^{*} A w] =
    2\sum_{i\ne j} (\Re A_{ij})^2.
    \label{eqn:EVarTrace}
  \end{equation}
\end{theorem}
\begin{proof}
  First 
  \[
  \mathbb{E} [w^{*} A w] = \mathbb{E}\left[\sum_{ij} w_{i} w_{j}
  A_{ij}\right] =
  \sum_{i} A_{ii} = \Tr[A].
  \]
  Here we used that $\mathbb{E} [w]= 0, \quad \mathbb{E} [w w^{*}] = I$
  and $w$ is a real vector. \REV{This follows directly from that each
  entry of $w$ is independently distributed and follows the Gaussian
  distribution $\mathcal{N}(0,1)$.}

  Second,
  \[
  \begin{split}
    \mathrm{Var} [w^{*} A w] = &\mathbb{E} \left[(w^{*} A w)^2 -
    (\Tr[A])^2\right]
  = \mathbb{E} \left[\sum_{ijkl} w_{i} w_{j} w_{k} w_{l} A_{ij}
  A_{kl} - (\Tr[A])^2\right]\\
  = & \sum_{ik} A_{ii} A_{kk} + 2 \sum_{i\ne j} (\Re A_{ij})^2 -
  (\sum_{i} A_{ii})^2 = 2 \sum_{i\ne j} (\Re A_{ij})^2.
  \end{split}
  \]
\end{proof}

Using Theorem~\ref{thm:EVarTrace}, if we choose $W\in
\mathbb{R}^{N\times N_{v}}$ to be a random Gaussian matrix, then 
\[
\Tr[g_{\sigma}(tI-A)] \approx \frac{1}{N_{v}}\Tr[W^{*}g_{\sigma}(tI-A)
W].
\]

In practice in order to compute $g_{\sigma}(A-tI) W$, we can expand
$g_{\sigma}(A-tI)$ into polynomials of $A$ for each $t$,
and then evaluate the trace of polynomial of $A$. A stable and efficient
implementation can be obtained by using Chebyshev polynomials.
Other choices of polynomials such as Legendre polynomials can be
constructed similarly. Using the Chebyshev
polynomial, $g_{\sigma}(tI-A)$ is approximated by a polynomial
of degree $M$ as
\begin{equation}
	g_{\sigma}(tI-A) \approx g^{M}_{\sigma}(tI-A) := \sum_{l=0}^{M} \mu_{l}(t) T_{l}(A).
  \label{eqn:gcheb}
\end{equation}
The coefficients $\{\mu_{l}(t)\}$ need to be evaluated for each
$t$.  Since the Dirac $\delta$-function is regularized using a Gaussian
function, following the notion in~\cite{LinSaadYang2015} we refer to
Eq.~\eqref{eqn:gcheb} as the ``Delta-Gauss-Chebyshev'' (DGC) expansion.
\REV{Since $g_{\sigma}(t-\cdot)$ is a smooth function on $(-1,1)$,}
the coefficient $\mu_{l}(t)$ in the DGC expansion can be computed as
\REV{
\begin{equation}
  \mu_{l}(t) = \frac{2-\delta_{l0}}{\pi} \int_{-1}^{1}
  \frac{1}{\sqrt{1-s^2}} g_{\sigma}(t-s) T_{l}(s) \ud s.
  \label{eqn:mucoef1}
\end{equation}
}
Here $\delta_{l0}$ is the Kronecker $\delta$ symbol. With change of
coordinate $s=\cos \theta$, and use the fact that
$T_{l}(s)=\cos(l\arccos(s))$, we have
\REV{
\begin{equation}
  \mu_{l}(t) = \frac{2-\delta_{l0}}{2\pi} \int_{0}^{2\pi}
  g_{\sigma}(t-\cos\theta) \cos(l\theta) \ud \theta.
  \label{eqn:mucoef2}
\end{equation}
}
Eq.~\eqref{eqn:mucoef2} can be evaluated by discretizing the interval 
$[0,2\pi]$ using a uniform grid, and the resulting quadrature can be
efficiently computed by
the Fast Fourier Transform (FFT). This procedure is given in
Alg.~\ref{alg:dgcexpansion}, \REV{and this procedure is usually inexpensive.} 

\begin{algorithm}[h]
\begin{small}
\begin{center}
  \begin{minipage}{5in}
\begin{tabular}{p{0.5in}p{4.5in}}
{\bf Input}:  &  \begin{minipage}[t]{4.0in}
  Chebyshev polynomial degree $M$;\\
  Number of integration points $2 N_{\theta}$, with $N_{\theta}>M$; \\
  Smooth function $g_{\sigma}(t-\cdot)$.
\end{minipage} \\
{\bf Output}:  &  \begin{minipage}[t]{4.0in}
  Chebyshev expansion coefficients \REV{$\{\mu_{l}(t)\}_{l=0}^{M}$}.
\end{minipage} 
\end{tabular}
\begin{algorithmic}[1]
  \STATE Let $\theta_{j}=\frac{j\pi}{N_{\theta}}, \quad
  j=0,\ldots,2N_{\theta}-1$.
  \STATE \REV{$g_{j}=g_{\sigma}(t-\cos\theta_{j})$}.
  \STATE Compute \REV{$\hat{g}=\mathcal{F} [g]$}, where $\mathcal{F}$ is the discrete 
  Fourier transform. Specifically
  \[
  \REV{\hat{g}_{l}} =  \sum_{j=0}^{2N_{\theta}-1}
  e^{-\frac{\I 2\pi jl}{2N_{\theta}}} \REV{g_{j}}.
  \]
  \STATE $\REV{\mu_{l}(t)}=\frac{2-\delta_{l0}}{2N_{\theta}}\Re
  \REV{\hat{g}_{l}},\quad
  l=0,\ldots,M$.
\end{algorithmic}
\end{minipage}
\end{center}
\end{small}
\caption{Computing the Delta-Gauss-Chebyshev (DGC) polynomial expansion
\REV{at a given point $t$}.}
\label{alg:dgcexpansion}
\end{algorithm}

Using the DGC expansion and Hutchinson's method, $\phi_{\sigma}(t)$ can
be approximated by
\[
\wt{\phi}_{\sigma}(t) := \Tr[g^{M}_{\sigma}(tI-A)] \approx  \sum_{l=0}^{M}
\mu_{l}(t) \frac{1}{N_{v}}\Tr[W^*T_{l}(A) W] \equiv
\sum_{l=0}^{M} \mu_{l}(t) \zeta_{l}.
\]
The resulting algorithm, referred to as the DGC algorithm in the
following, is given in Alg.~\ref{alg:dgcDOS}.


\begin{algorithm}  
\begin{small}
\begin{center}
  \begin{minipage}{5in}
\begin{tabular}{p{0.5in}p{4.5in}}
{\bf Input}:  &  
  \begin{minipage}[t]{4.0in}
    Hermitian matrix $A$ with eigenvalues between $(-1,1)$;\\ 
    A set of points $\{t_{i}\}_{i=1}^{N_{t}}$ at which the DOS is to be evaluated;\\
    Polynomial degree $M$; Smearing parameter $\sigma$;\\
    Number of random vectors $N_{v}$.\\
  \end{minipage}\\
  {\bf Output}: &  Approximate DOS
  $\{\wt{\phi}_{\sigma}(t_i)\}_{i=1}^{N_{t}}$. \\
\end{tabular}
\begin{algorithmic}[1]
\FOR {each $t_i$}
\STATE Compute the coefficient $\{\mu_{l}(t_{i})\}_{l=0}^{M}$ for each
$t_{i},i=1,\ldots,N_{t}$ using Alg.~\ref{alg:dgcexpansion}.
\ENDFOR
\STATE Initialize $\zeta_k = 0 $ for $k=0,\cdots, M$.
\STATE Generate a random Gaussian matrix $W\in \mathbb{R}^{N\times
N_{v}}$.
\STATE Initialize the three term recurrence matrices 
$V_{m},V_{p}\gets \mathbf{0}\in \mathbb{C}^{N\times N_{v}},V_{c}\gets
W$.
\FOR {$l=0,\ldots,M$}
\STATE Accumulate $\zeta_{l} \gets \zeta_{l}+ \frac{1}{N_{v}} \Tr[W^{*}
V_{c}]$.
\STATE $V_{p}\gets (2-\delta_{l0}) A V_{c} - V_{m}$.
\STATE $V_{m}\gets V_{c}, V_{c}\gets V_{p}$.
\ENDFOR
\FOR {$i=1,\ldots, N_{t}$}
\STATE Compute $\wt{\phi}_{\sigma}(t_{i}) \gets \sum_{l=0}^{M} \mu_{l}(t_{i}) \zeta_{l}$.
\ENDFOR

\end{algorithmic}
\end{minipage}
\end{center}
\end{small}
\caption{The Delta-Gauss-Chebyshev (DGC) method for estimating the DOS.}
\label{alg:dgcDOS}
\end{algorithm}

We remark that the DGC algorithm can be viewed as a variant of the
kernel polynomial method (KPM)~\cite{WeiseWelleinAlvermannEtAl2006}.
The difference is that KPM \textit{formally} expands the Dirac
$\delta$-function, which is not a well defined function but only a
distribution.  Therefore the accuracy of KPM cannot be properly measured
until regularization is introduced~\cite{LinSaadYang2015}.  On the other
hand, DGC introduces a Gaussian regularization from the beginning, and
the DGC expansion~\eqref{eqn:gcheb} is not a formal expansion, and its
accuracy can be \REV{relatively easily analyzed.} 
\REV{The proof of the accuracy of the DGC expansion can be obtained via
the same techniques used in
e.g.~\cite{Demko1977,BenziBoitoRazouk2013,Lin2015}, and is given below for completeness.}

\REV{
Let $k$ be a non-negative integer, and $\mathbb{P}_{k}$ be the set of
all polynomials of degrees less than or equal to $k$ with real
coefficients. For a real continuous function $f$
on $[-1,1]$, the best approximation error is defined as
\begin{equation}
  E_{k}(f) = \REV{\min_{p\in \mathbb{P}_{k}}}\left\{
  \norm{f-p}_{\infty}:=
  \max_{-1\le x\le 1}\abs{f(x)-p(x)}\right\}.
  \label{eqn:approxerror}
\end{equation}
It is known that such best approximation error is achieved by Chebyshev
polynomials~\cite{Meinardus1967}.
Consider an ellipse in the complex plane $\mathbb{C}$ with foci in $-1$
and $1$, and $a>1,b>0$ be the half axes so that the vertices of the
ellipse are $a,-a, ib, -ib$, respectively. Let the sum of the half axes
be $\chi=a+b$, then using the identity $a^2-b^2=1$ we have
\[
a = \frac{\chi^2+1}{2\chi}, \quad b = \frac{\chi^2-1}{2\chi}.
\]
Thus the ellipse is determined only by $\chi$, and such ellipse is
denoted by $\mathcal{E}_{\chi}$.
Then Bernstein's theorem~\cite{Meinardus1967} is stated  as follows.
\begin{theorem}[Bernstein]
  Let $f(z)$ be analytic in $\mathcal{E}_{\chi}$ with $\chi>1$, and
  $f(z)$ is a real valued function for real $z$. Then
  \begin{equation}
    E_{k}(f) \le \frac{2 M(\chi)}{\chi^{k} (\chi-1)},
    \label{eqn:bernstein}
  \end{equation}
  where
  \begin{equation}
    M(\chi) = \sup_{z\in \mathcal{E}_{\chi}}\abs{f(z)}.
    \label{eqn:Mchi}
  \end{equation}
  \label{thm:bernstein}
\end{theorem}
}

\REV{Using Theorem~\ref{thm:bernstein}, a quantitative description of
the approximation properties for Gaussian functions is given in
Theorem~\ref{thm:errordgc}.}

\begin{theorem}
  Let $A\in \mathbb{C}^{N\times N}$ be a Hermitian matrix with spectrum
  in $(-1,1)$.  For any $t\in \mathbb{R}$, the error of a $M$-term DGC expansion~\eqref{eqn:gcheb} is
  \begin{equation}
    \abs{\Tr[g_{\sigma}(tI-A)]-\Tr[g^{M}_{\sigma}(tI-A)]}
    \le \frac{C_{1}}{\sigma} (1+C_{2}\sigma)^{-M},
    \label{eqn:chebyaccuracy}
  \end{equation}
  \label{thm:errordgc}
  where $C_{1},C_{2}$ are constants independent of \REV{$A$ as well as
  $\sigma,M,t$}.
\end{theorem}
\begin{proof}
  \REV{For any $t\in \mathbb{R}, \sigma>0$, the Gaussian function
  $g_{\sigma}(t-\cdot)$ is analytic in any ellipse $\mathcal{E}_{\chi}$ with
  $\chi>1$, then
  \[
  M(\chi)=\sup_{z= x+iy\in
  \mathcal{E}_{\chi}}\abs{g_{\sigma}(t-(x+iy))}  \le 
  \frac{1}{N}\sup_{z= x+iy\in \mathcal{E}_{\chi}} 
  e^{\frac{y^2}{2\sigma^2}}\le 
  \frac{1}{N}
  e^{\frac{(\chi-\frac{1}{\chi})^2}{8\sigma^2}}.
  \]
  For any $\alpha>0$, let 
  \begin{equation}
    \chi=1+\alpha \sigma,
    \label{eqn:chichoice}
  \end{equation}
  then
  $\chi-\frac{1}{\chi} \le 2\alpha \sigma$, and
  \begin{equation}
    M(1+\alpha \sigma) \le \frac{1}{N} e^{\frac{\alpha^2}{2}}.
    \label{eqn:Mbound}
  \end{equation}
  Then the error estimate follows from Theorem~\ref{thm:bernstein} that
  \[
  E_{M}(g_{\sigma}(t-\cdot)) \le \frac{2}{N\alpha\sigma}
  e^{\frac{\alpha^2}{2}}
    (1+\alpha\sigma)^{-M}.
  \]
  Finally
  \[
  \begin{split}
  &\abs{\Tr[g_{\sigma}(tI-A)]-\Tr[g^{M}_{\sigma}(tI-A)]} \\
  \le& 
  N \norm{g_{\sigma}(tI-A)-g^{M}_{\sigma}(tI-A)}_{2}\\
  = & N E_{M}(g_{\sigma}(t-\cdot))  
  \le \frac{2}{\alpha\sigma} e^{\frac{\alpha^2}{2}}
    (1+\alpha\sigma)^{-M}.
  \end{split}
  \]
  The theorem is then proved by defining
  $C_{1}=\frac{2}{\alpha} e^{\frac{\alpha^2}{2}}$, $C_{2}=\alpha$.
  Since $\alpha$ can be chosen to be any constant due to the analyticity of the
  Gaussian function in the complex plane, both $C_{1}$ and $C_{2}$ are independent of
  $A$ as well as $\sigma,M,t$.
  }
\end{proof}

\REV{Theorem~\ref{thm:errordgc} indicates that }the error of the DGC
algorithm is split into two parts: the error of the Chebyshev expansion
(approximation error) and the error due to random sampling (sampling
error). \REV{According to Theorem~\ref{thm:errordgc}, it} is sufficient to choose $M$ to be
\REV{$\Or(\sigma^{-1}\abs{\log\sigma})$} to
ensure that the error of the Chebyshev expansion is negligible.
Therefore the error of the DGC mainly comes from the sampling error,
which decays slowly as $\frac{1}{\sqrt{N_{v}}}$.


\section{Spectrum sweeping method for estimating spectral densities}\label{sec:doslowrank}

In this section we present an alternative randomized algorithm called
the spectrum sweeping method for
estimating spectral densities.  
\REV{Our numerical results indicate that the spectrum sweeping method
can significantly outperform Hutchinson type methods in terms of
accuracy,  as the number of random vectors $N_{v}$ becomes large.}
The main tool is randomized methods for low
rank matrix decomposition which is briefly introduced in
section~\ref{subsec:randomsvd}.  The spectrum sweeping method is given in
section~\ref{subsec:lralg}, and its complexity is analyzed in
section~\ref{subsec:sscomplexy}.

\subsection{Randomized method for \REV{low rank decomposition of a
numerically low rank matrix}}
\label{subsec:randomsvd}

\REV{Consider a square matrix $P\in \mathbb{C}^{N\times N}$, and denote
by $r$ the rank
of $P$. If $r\ll N$, then $P$ is called a low rank matrix. Many matrices from scientific
and engineering computations may not be exactly low rank but
are close to be a low rank matrix. For such matrices, the concept of
\textit{numerical rank} or \textit{approximate rank} can be introduced,
defined by the closest matrix to $P$ in the sense of the matrix
$2$-norm. More specifically, the numerical $\varepsilon$-rank of
a matrix $P$, denoted by $r_{\varepsilon}$ with respect to the tolerance 
$\varepsilon>0$ is~(see e.g. \cite{GolubVan2013} section 5.4)
\[
r_{\varepsilon} = \min\{\mathrm{rank}(Q):Q\in\mathbb{C}^{N\times N},
\norm{P-Q}_{2}\le \varepsilon\}.
\]
In the following discussion, we simply refer to a matrix
$P$ with numerical $\varepsilon$-rank $r_{\varepsilon}$ as a matrix with
numerical rank $r$. For instance, this applies to the function
$g_{\sigma}(tI-A)$ with small $\sigma$, since the value
$g_{\sigma}(t-\lambda_{j})$ decays fast to $0$ when $\lambda_{j}$ is away from
$t$, and the corresponding contribution to the rank of $g_{\sigma}(tI-A)$
can be neglected up to $\varepsilon$ level. As an example, for the
ModES3D\_4 matrix to be detailed in section~\ref{sec:numer}, if we set
$\sigma=0.01$ and $t=1.0$, then the values
$\{g_{\sigma}(tI-\lambda)\}$
sorted in non-increasing order is given in Fig.~\ref{fig:eigA}, where each
$\lambda$ is an eigenvalue of $A$. If we set
$\varepsilon=10^{-8}$, then the $\varepsilon$-rank of $g_{\sigma}(tI-A)$
is $59$, much smaller than the dimension of the matrix $A$ which is
$64000$.}

\begin{figure}[h]
  \begin{center}
    \includegraphics[width=0.3\textwidth]{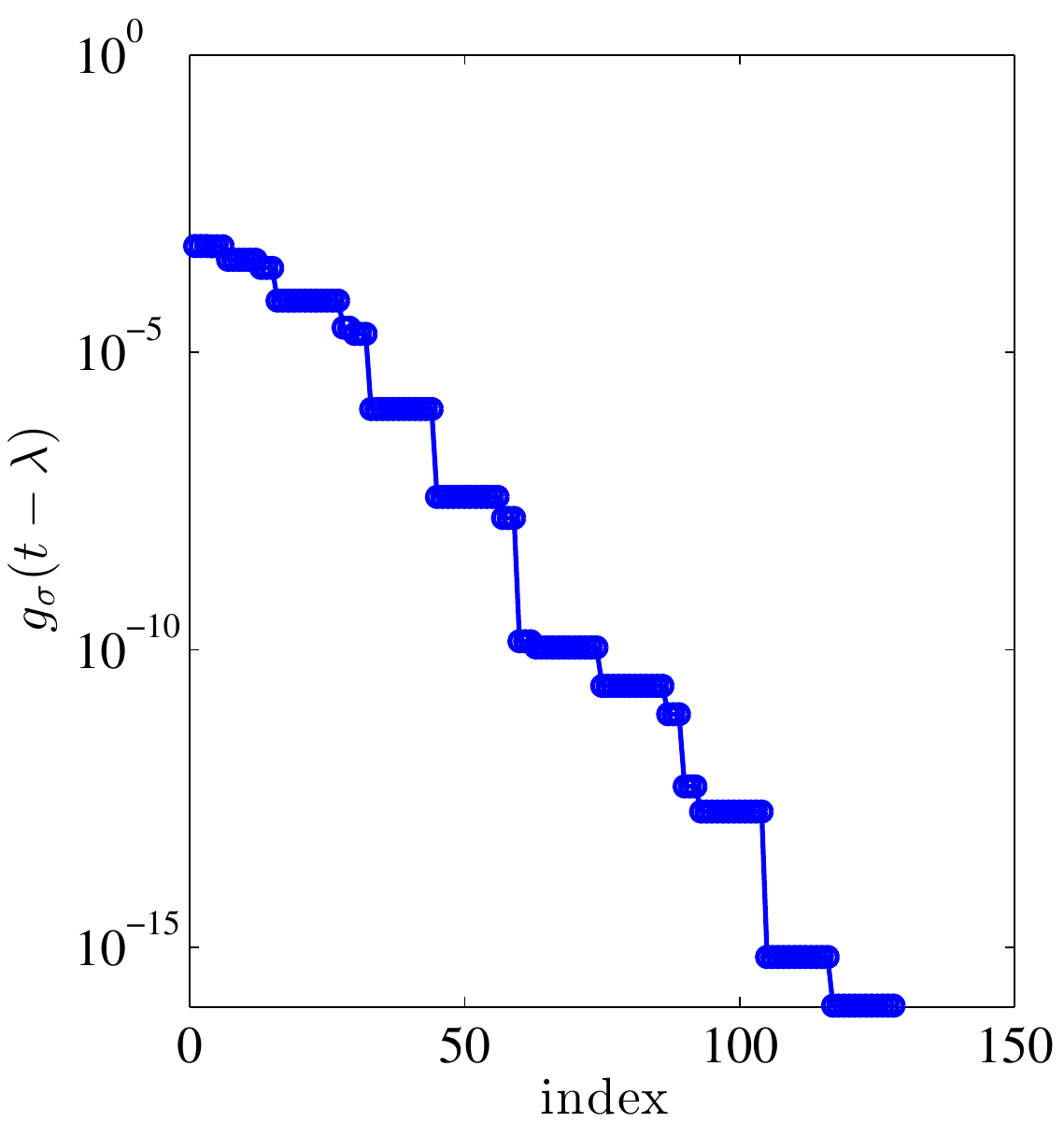}
  \end{center}
  \caption{\REV{For the ModES3D\_4 matrix, the values
  $g_{\sigma}(tI-\lambda)$ sorted in non-increasing order plotted in log
  scale with $\sigma=0.01,t=1.0$. Only the first $128$ values larger
  than $10^{-16}$ are shown.}}
  \label{fig:eigA}
\end{figure}

For a \REV{numerically} low rank matrix, its \REV{approximate} singular value
decomposition can be efficiently evaluated using randomized algorithms
(see
e.g.~\cite{LibertyWoolfeMartinssonEtAl2007,WoolfeLibertyRokhlinEtAl2008,HalkoMartinssonTropp2011}).
The idea is briefly reviewed as below, though presented in a slightly
non-standard way.  \REV{Let $P\in \mathbb{C}^{N\times N}$ be a square
matrix with numerical rank $r \ll N$}, and $W\in \mathbb{R}^{N\times N_{v}}$ be a
random Gaussian matrix. If $N_{v}$ is larger than $r$ by a small
constant, then with high probability, \REV{$P$ projected to the column
space of $PW$ is very close to $P$ in matrix $2$-norm}.
Similarly with high probability, \REV{$P$ projected to the row space
of $W^{*}P$ is very close to $P$ in matrix $2$-norm.}
In the case when $P$ is a Hermitian matrix, only the matrix-vector
multiplication $PW$ is needed.

In order to construct an approximate low rank decomposition of a
Hermitian matrix $P$,
let us denote by $Z=PW$, then an approximate low rank decomposition of
$P$ is given by
\REV{
\begin{equation}
  P \approx ZBZ^{*}.
  \label{eqn:Papprox}
\end{equation}
The matrix $B$ is to be determined and can be computed in several
ways.  One choice of $B$ \REV{can be obtained by requiring
Eq.~\eqref{eqn:Papprox} to hold when applying $W^{*}$ and $W$ to the
both sides of the equation, i.e.}
\[
  K_{W} := W^{*}PW = W^{*} Z \approx (W^{*}Z) B (W^{*}Z)^{*} = (W^{*}Z) B (W^{*}Z).
\]
In the last equality we used that $(W^{*}Z)$ is Hermitian. Hence one can choose
\begin{equation}
  B = (W^{*}Z)^{\dagger} \equiv K_{W}^{\dagger}.
  \label{eqn:Bchoice}
\end{equation}
}
Here $K_{W}^{\dagger}$ is the Moore-Penrose
pseudo-inverse (see e.g.~\cite{GolubVan2013}, section 5.5) of the
matrix $K_{W}$.   In Alg.~\ref{alg:ranlowrank} we summarize the algorithm
for constructing such a low rank decomposition.  

%

\begin{remark}
  In the case when $K_{W}$ is singular, the pseudo-inverse should be
  handled with care.  \REV{This will be discussed in
  section~\ref{subsec:lralg}.
  Besides the choice in Eq.~\eqref{eqn:Bchoice}, another possible choice of the matrix
  $B$ is given by requiring that Eq.~\eqref{eqn:Papprox} holds when
  applying $Z^{*}$ and $Z$ to the both sides of the equation, i.e.
  \[
  Z^{*} P Z \approx (Z^{*}Z) B (Z^{*}Z),
  \]
  and hence one can choose 
  \begin{equation}
    B = (Z^{*}Z)^{\dagger} (Z^{*}PZ) (Z^{*}Z)^{\dagger}.
    \label{eqn:Bchoice2}
  \end{equation}
  Note that Eq.~\eqref{eqn:Bchoice2} can also be derived from a
  minimization problem 
  \[
  \min_{B} \norm{ZBZ^{*}-P}_{F}^2.
  \]
  However, the
  evaluation of $Z^{*}PZ$ is slightly more difficult to compute, since
  the matrix-vector multiplication $PZ$ needs to be further computed.
  In the discussion below
  we will adopt the choice of Eq.~\eqref{eqn:Bchoice}.
  }
\end{remark}

\begin{algorithm}  
\begin{small}
\begin{center} \begin{minipage}{5in} \begin{tabular}{p{0.5in}p{4.5in}} {\bf Input}:  &  \begin{minipage}[t]{4.0in}
  Hermitian matrix $P\in\mathbb{C}^{N\times N}$ with approximate rank $r$;\\
\end{minipage} \\
{\bf Output}:  &  \begin{minipage}[t]{4.0in}
  Approximate low rank decomposition $P\approx ZBZ^{*}$.
\end{minipage} 
\end{tabular}
\begin{algorithmic}[1]
  \STATE Generate a random Gaussian matrix $W\in \mathbb{R}^{N\times
  N_{v}}$ where $N_{v}=r+c$ and $c$ is a small constant.
  \STATE Compute $Z \gets P W$.
  \STATE Form $B = (W^{*}Z)^{\dagger}$.
\end{algorithmic}
\end{minipage}
\end{center}
\end{small}
\caption{Randomized low rank decomposition algorithm.}
\label{alg:ranlowrank}
\end{algorithm}


\subsection{Spectrum sweeping method}\label{subsec:lralg}

The approximate low rank decomposition method can be used to estimate the DOS. Note
that for each $t$, when the regularization parameter $\sigma$ is small
enough, the column space of $g_{\sigma}(tI-A)$ is approximately only
spanned by eigenvectors of $A$ corresponding to eigenvalues near $t$.
Therefore for each $t$, $g_{\sigma}(tI-A)$ is approximately a low rank
matrix. Alg.~\ref{alg:ranlowrank} can be used to construct a low rank
decomposition. \REV{Motivated from the DGC method, we can use the same
random matrix $W$ for all $t$.}
\[
g_{\sigma}(tI-A) \approx Z(t) (W^{*}Z(t))^{\dagger} Z^{*}(t),
\]
and its trace can be \REV{accurately} estimated as
\begin{equation}
  \Tr[g_{\sigma}(tI-A)] \approx  \Tr[(W^{*}Z(t))^{\dagger}
  (Z^{*}(t)Z(t))].
  \label{eqn:tracelr}
\end{equation}
\REV{Here $Z(t)=g^{M}_{\sigma}(tI-A) W$}. We can use a Chebyshev
expansion in Eq.~\eqref{eqn:gcheb} \REV{and compute $g^{M}_{\sigma}(tI-A)$}.  

The Chebyshev expansion requires the calculation of $T_{l}(A)
W,\quad l=0,\ldots,M$. Note
that this does not mean that all $T_{l}(A) W$ need to be stored for all
$l$.  Instead we only need to accumulate $Z(t)$ for each point $t$ that
the DOS is to be evaluated.  $T_{l}(A)$ only need to be applied
to one random $W$ matrix, and we can sweep through the spectrum
of $A$ just via different linear combination of all $T_{l}(A) W$ for
each $t$.  Therefore we refer to the algorithm a ``spectrum sweeping''
method.


As remarked earlier, the pseudo-inverse should be handled with care.
\REV{There are two difficulties associated with the evaluation of $K_{W}^{\dagger}$. First,
it is difficult to know \textit{a priori} the exact number of vectors
$N_{v}$ that should be used at each $t$, and $N_{v}$ should be chosen to
be large enough to achieve an accurate estimation of the DOS.  Hence
the columns of $Z$ are likely to be nearly linearly dependent, and
$K_{W}$ becomes singular. Second, although $g_{\sigma}(tI-A)$ is
by definition a positive semidefinite matrix, the finite term Chebyshev
approximation $g^{M}_{\sigma}(tI-A)$ may not be positive semidefinite
due to the oscillating tail of the Chebyshev polynomial.  
Fig.~\ref{fig:PinvEx} (a) gives an example of such possible failure.
The test matrix is the
ModES3D\_1 matrix to be detailed in section~\ref{sec:numer}. The
parameters are $\sigma=0.05,N_{v}=50$. When computing the
pseudo-inverse, all negative eigenvalues and positive eigenvalues with
magnitude less than $10^{-7}$ times the largest eigenvalue of $K_{W}$
are discarded.  Fig.~\ref{fig:PinvEx} (a) demonstrates that when a
relatively small number of degrees of polynomials $M=400$ is used, 
the treatment of the pseudo-inverse may have large error 
near $t=0.9$. 
This happens mainly when the degrees
of Chebyshev polynomials $M$ is not large enough.  Fig.~\ref{fig:PinvEx}
(b) shows that when $M$ is increased to $800$, the accuracy of the
pseudo-inverse treatment is much
improved.}

\begin{figure}[h]
  \begin{center}
    \subfloat[(a)]{\includegraphics[width=0.4\textwidth]{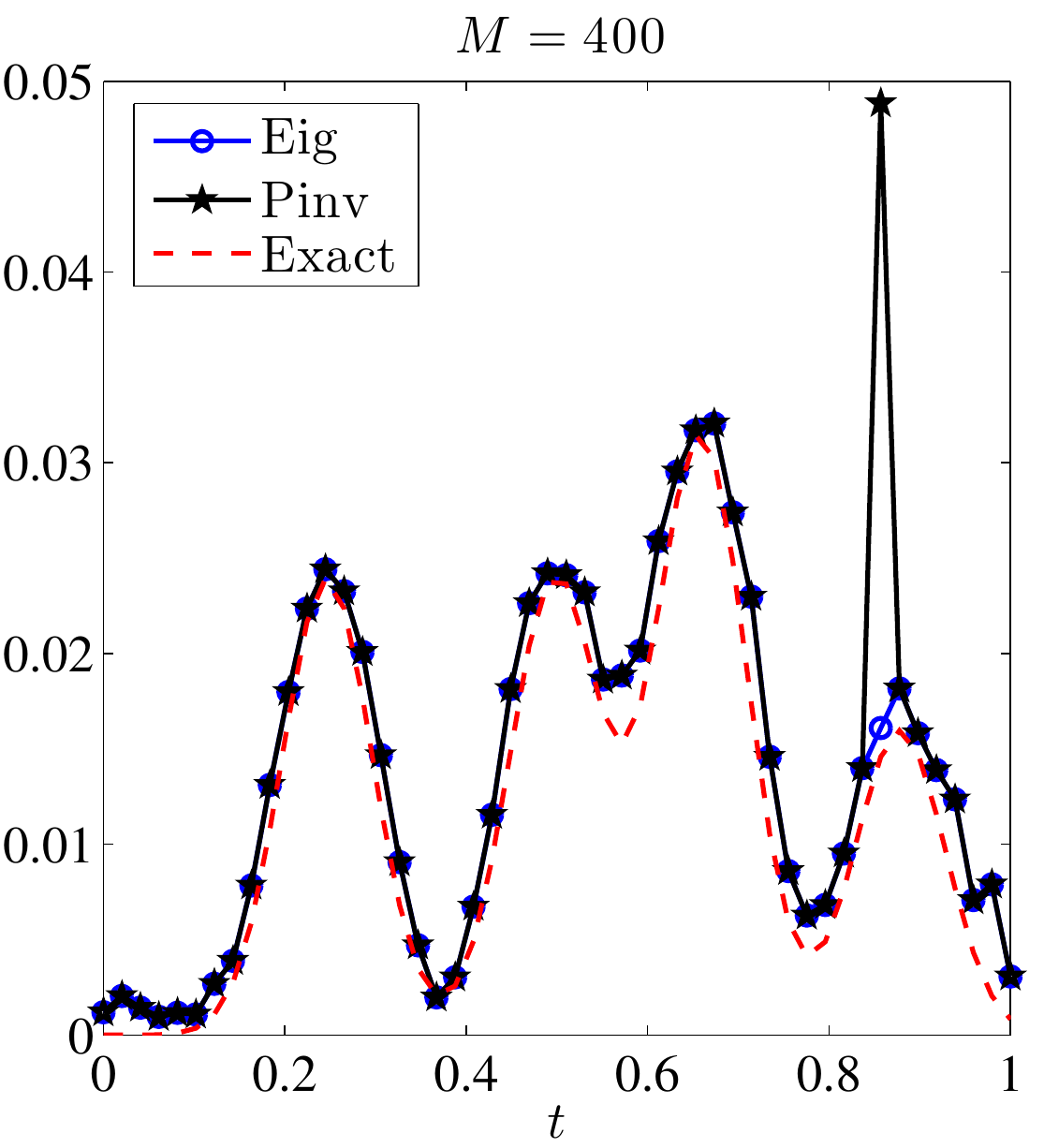}}
    \quad
    \subfloat[(b)]{\includegraphics[width=0.4\textwidth]{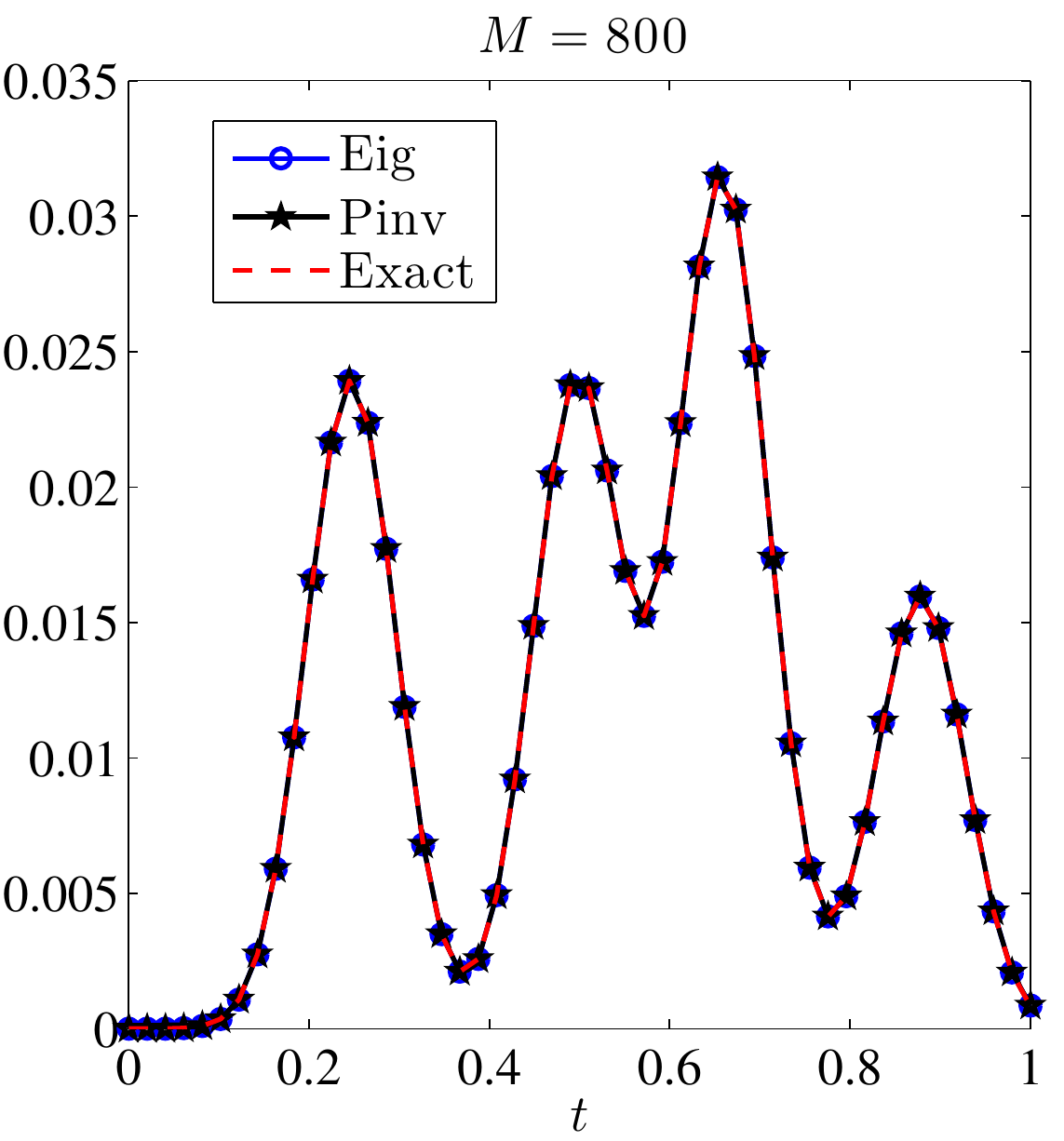}}
  \end{center}
  \caption{For the ModES3D\_1 matrix, compute the DOS using the
  spectrum sweeping method by computing the pseudo-inverse (``Pinv'')
  and by computing the generalized eigenvalue problem (``Eig'') using (a)
  a low degree Chebyshev polynomial $M=400$ (b) a high degree Chebyshev polynomial
  $M=800$.}
  \label{fig:PinvEx}
\end{figure}

\REV{The possible difficulty of the direct evaluation of $K_{W}^{\dagger}$
can be understood as follows. First, Theorem~\ref{thm:correspondence} suggests that for a
strictly low rank matrix $P$, there is correspondence between the trace of
the form in Eq.~\eqref{eqn:tracelr} and the solution of a generalized
eigenvalue problem.}

\begin{theorem}\label{thm:correspondence}
  Let $P\in \mathbb{C}^{N\times N}$ be a Hermitian matrix with rank
  $r \ll N$ and with eigen decomposition
  \[
  P = U S U^{*}.
  \]
  Here $U\in\mathbb{C}^{N\times r}$ and $U^{*} U = I$.
  $S\in\mathbb{R}^{r\times r}$ is a real
  diagonal matrix containing the nonzero eigenvalues of $P$. For
  $W\in\mathbb{C}^{N\times p}$ ($p>r$) and assume $W^{*}U$ is a matrix
  with linearly independent columns, then $S$ can be recovered through the generalized
  eigenvalue problem
  \begin{equation}
    (W^{*}P^{2} W) C = (W^{*}P W) C \REV{\Xi}.
    \label{eqn:generalev}
  \end{equation}
  Here $C\in \mathbb{C}^{p\times r}$ and \REV{$\Xi\in \mathbb{R}^{r\times
  r}$} is a diagonal matrix with diagonal entries equal to those of $S$
  up to reordering.  Furthermore, 
  \begin{equation}
    \Tr\left[ (W^{*}PW)^{\dagger}(W^{*}P^2W) \right] = \Tr[S] =
    \Tr[P].
    \label{eqn:generaltrace}
  \end{equation}
\end{theorem}
\begin{proof}
  Using the eigen decomposition of $P$,
  \[
  (W^{*}P^{2} W) C  = (W^{*} U) S^{2} (U^{*} W C), \quad (W^{*}P W) C
  \REV{\Xi} =
  (W^{*}U) S (U^{*} W C) \REV{\Xi}.
  \]
  Since $W^{*}U\in \mathbb{C}^{p\times r}$ is a matrix with linearly
  independent columns, we have
  \[
  S^{2} (U^{*}W C) = S (U^{*}W C)  \REV{\Xi},
  \]
  or equivalently
  \[
  S (U^{*}W C) = (U^{*}W C) \REV{\Xi}.
  \]
  Since $S$ is a diagonal matrix we have $S=\Xi$ up to reordering of
  diagonal entries.

  To prove Eq.~\eqref{eqn:generaltrace}, note that
	\[
	(W^{*} PW)^{\dagger} = (U^{*}W)^{\dagger} S^{-1}(W^{*}U)^{\dagger}.
	\]
	Therefore
	\[
  \begin{split}
	\Tr\left[(W^{*}PW)^{\dagger}(W^{*}P^2W)\right] =&
	\Tr\left[(U^{*}W)^{\dagger} S^{-1}(W^{*}U)^{\dagger}
  (W^{*}U) S^{2} (U^{*}W)\right] \\
  =& \Tr[S] = \Tr[USU^{*}] = \Tr[P].
  \end{split}
	\]
\end{proof}

\REV{
Although Theorem~\ref{thm:correspondence} is stated for exactly low rank
matrices, it provides an practical criterion for removing some of
the large, spurious contribution to the DOS such as in
Fig.~\ref{fig:PinvEx} (a). 
For $g_{\sigma}(tI-A)$ which is numerically
low rank, only the generalized eigenvalues within the range of
$g$, i.e. the interval $[0,\frac{1}{N\sqrt{2\pi\sigma^2}}]$ should be
selected.  This motivated the use of Alg.~\ref{alg:geneig} to solve the
generalized eigenvalue problem 
\begin{equation}
  Z^{*}Z \wt{C} = (W^{*}Z) \wt{C} \wt{\Xi}.
  \label{eqn:geneig}
\end{equation}
where $\wt{\Xi}$ is a diagonal matrix only containing the generalized
eigenvalues in the interval $[0,\frac{1}{N\sqrt{2\pi\sigma^2}}]$.
Alg.~\ref{alg:geneig} only keeps the generalized eigenvalues within the
possible range of $g_{\sigma}$. Our numerical experience indicates that
this procedure is more stable than the direct treatment of the
pseudo-inverse.  The algorithm of the spectrum sweeping method is given
in Alg. ~\ref{alg:rlrdos}.}

\REV{Now consider the problematic point when using the pseudo-inverse in
Fig.~\ref{fig:PinvEx} (a).  We find that the generalized eigenvalues
$\Xi$  at that problematic point has one generalized eigenvalue $0.033$, exceeding
the maximally allowed range $\frac{1}{N\sqrt{2\pi\sigma^2}}=0.008$.
After removing this generalized eigenpair, the error of the DOS obtained by
solving the generalized eigenvalue problem becomes smaller. Again when
the degree of the Chebyshev polynomial $M$ increases sufficiently large to $800$, the
accuracy of the generalized eigenvalue formulation also improves, and the result obtained by
using the pseudo-inverse and by using the generalized eigenvalue problem
agrees with each other, as illustrated in Fig.~\ref{fig:PinvEx} (b). In
such case, all generalized eigenvalues fall into the
range $[0,\frac{1}{N\sqrt{2\pi\sigma^2}}]$.}

\begin{algorithm}  
\begin{small}
\begin{center}
  \begin{minipage}{5in}
\begin{tabular}{p{0.5in}p{4.5in}}
{\bf Input}:  &  \begin{minipage}[t]{4.0in}
  Matrices $Z,W\in \mathbb{C}^{N\times N_{v}}$;\\
  Smearing parameter $\sigma$;\\
  Truncation parameter $\tau$.
\end{minipage} \\
{\bf Output}:  &  \begin{minipage}[t]{4.0in}
  Generalized eigenvalues $\wt{\Xi}$ and generalized eigenvectors
  $\wt{C}$.
\end{minipage} 
\end{tabular}
\begin{algorithmic}[1]
  \STATE Compute the eigenvalue decomposition of the matrix
  \[
  W^{*}Z = USU^{*}.
  \]
  $S$ is a diagonal matrix with diagonal entries $\{s_{i}\}$.
  \STATE Let $\wt{S}$ be a diagonal matrix with all eigenvalues
  $s_{j}\ge \tau \max_{i} s_{i}$, and $\wt{U}$ be the corresponding eigenvectors.
  \STATE Solve the standard eigenvalue problem
  \[
  \wt{S}^{-\frac12} \wt{U}^{*} Z^{*} Z \wt{U} \wt{S}^{-\frac12} X = X
  \Xi.
  \]
  $\Xi$ is a diagonal matrix with diagonal entries $\{\xi_{i}\}$.
  \STATE Let $\wt{\Xi}$ be a diagonal matrix containing the generalized
  eigenvalues $\xi_{i}\in [0,\frac{1}{N\sqrt{2\pi\sigma^2}}]$, and
  $\wt{X}$ be the corresponding eigenvectors.
  \STATE Compute the generalized eigenvectors
  $\wt{C}=\wt{U}\wt{S}^{-\frac12}\wt{X}$.
\end{algorithmic}
\end{minipage}
\end{center}
\end{small}
\caption{\REV{Solve the generalized eigenvalue problem for the spectrum sweeping
method.}}
\label{alg:geneig}
\end{algorithm}

\begin{algorithm}  
\begin{small}
\begin{center}
  \begin{minipage}{5in}
\begin{tabular}{p{0.5in}p{4.5in}}
{\bf Input}:  &  \begin{minipage}[t]{4.0in}
  Hermitian matrix $A$ with eigenvalues between $(-1,1)$;\\
  A set of points $\{t_{i}\}_{i=1}^{N_{t}}$ at which the DOS is to be evaluated;\\
  Polynomial degree $M$;
  Smearing parameter $\sigma$;\\
  Number of random vectors $N_{v}$.
\end{minipage} \\
{\bf Output}:  &  \begin{minipage}[t]{4.0in}
  Approximate DOS $\{\wt{\phi}_{\sigma}(t_i)\}$.
\end{minipage} 
\end{tabular}
\begin{algorithmic}[1]
  \STATE Compute the coefficient $\{\mu_{l}(t_{i})\}_{l=0}^{M}$ for each
  $t_{i},i=1,\ldots,N_{t}$ using Alg.~\ref{alg:dgcexpansion} with
  $f(x)= g_{\sigma}(x-t_i)$.
  \STATE Generate a random Gaussian matrix $W\in \mathbb{R}^{N\times
  N_{v}}$.
  \STATE Initialize the three term recurrence matrices 
  $V_{m},V_{p}\gets \mathbf{0}\in \mathbb{C}^{N\times N_{v}},V_{c}\gets
  W$.
  \STATE Initialize $Z(t_{i})\gets \mathbf{0}\in \mathbb{C}^{N\times
  N_{v}}, i=1,\ldots,N_{t}$.
  \FOR {$l=0,\ldots,M$}
  \FOR {$i=1,\ldots, N_{t}$}
  \STATE $Z(t_{i})\gets Z(t_i) + \mu_{l}(t_{i}) V_{c}$.
  \ENDFOR
  \STATE $V_{p}\gets (2-\delta_{l0}) A V_{c} - V_{m}$.
  \STATE $V_{m}\gets V_{c}, V_{c}\gets V_{p}$.
  \ENDFOR
  \FOR {$i=1,\ldots, N_{t}$}
  \STATE \REV{Compute $\wt{\phi}_{\sigma}(t_{i}) = \Tr[\wt{\Xi}(t_i)]$ where
  $\wt{\Xi}(t_i)$ is a diagonal matrix obtained by solving the generalized
  eigenvalue problem 
  \[
  Z^{*}(t_{i}) Z(t_{i}) \wt{C}(t_i) = W^{*}(t_{i}) Z(t_{i}) \wt{C}(t_i)
  \wt{\Xi}(t_i)
  \]
  using Alg.~\ref{alg:geneig}}.
  \ENDFOR
\end{algorithmic}
\end{minipage}
\end{center}
\end{small}
\caption{Spectrum sweeping method using the Delta-Gauss-Chebyshev expansion (SS-DGC) for estimating the DOS.} \label{alg:rlrdos}
\end{algorithm}

\REV{The SS-DGC method can be significantly more accurate compared to the DGC method.
This is because the spectrum sweeping method takes
advantage that different columns of $Z\equiv
g_{\sigma}(tI-A)W$ are correlated: The information in different columns
of $Z$ saturates as $N_{v}$ increases beyond the numerical rank of
$g_{\sigma}(tI-A)$, and the columns 
of $Z$ become increasingly linearly dependent.
Comparatively Hutchinson's method neglects such correlated information,
and the asymptotic convergence rate is only $\Or(N_{v}^{-\frac12})$.
}
%

\subsection{Complexity}\label{subsec:sscomplexy}

The complexity of the DOS estimation is certainly problem dependent.  In
order to measure the asymptotic complexity of Alg.~\ref{alg:rlrdos}
for a series of matrices with increasing value of $N$, we consider
a series of matrices are \textit{spectrally uniformly distributed}, i.e.
the spectral width of each matrix is bounded between $(-1,1)$, and the number of
eigenvalues in any interval $(t_{1},t_{2})$ is proportional to $N$. In
other words, we do not consider the case when the eigenvalues can
asymptotically be concentrated into one point. \REV{In the complexity
analysis below, we neglect any contribution on the order of $\log N$.}
In section~\ref{sec:robustdoslowrank} we will give a detailed example
for which the assumption is approximately satisfied. 

Alg.~\ref{alg:rlrdos} scales as $\Or(N_{v}^3)$ with respect to the
number of random vectors $N_{v}$.  Therefore it can be less efficient to
let $N_{v}$ grow proportionally to the matrix size $N$.  Instead it is
possible to choose $N_{v}$ to be a constant, and to choose the regularization
parameter $\sigma$ to be $\Or(N^{-1})$ so that $g_{\sigma}(tI-A)$ is a
matrix of bounded \REV{numerical} rank as $N$ increases.  Eq.~\eqref{eqn:chebyaccuracy}
then states that the Chebyshev polynomial degree $M$ should be $\Or(N)$.
We denote by $c_{\mathrm{matvec}}$ the cost of each matrix-vector
multiplication (matvec). We assume $c_{\mathrm{matvec}}\sim \Or(N)$.  We
also assume that $N_{v}$ is kept to be a constant and is omitted in the
asymptotic complexity count with respect to $N$  and $N_{t}$.

Under the assumption of spectrally uniformly
distributed matrices, the computational cost for applying the Chebyshev
polynomial to the random matrix $W$ is $c_{\mathrm{matvec}} M N_{v}\sim
\Or(N^2)$.  The cost for updating all $\{Z(t_{i})\}$ is $N_{t} M N
N_{v}\sim \Or(N^2 N_{t})$.  The cost for computing the DOS by trace
estimation is $\Or(N_{t} N N_{v}^2 + N_{t} N_{v}^3)\sim \Or(N_{t} N)$.
So the total cost is $\Or( N^2 N_{t} )$.

The memory cost is dominated by the storage of the matrices
$\{Z(t_{i})\}$, which scales as $\Or(N N_{t} N_{v})\sim \Or(N N_{t})$.  

If the DOS is evaluated at a few points with $N_{t}$ being small, the
spectrum sweeping method is very efficient.  However, in some cases such
as the trace estimation as discussed in section~\ref{sec:traceest},
$N_{t}$ should be chosen to be $\Or(N)$.  Therefore the computational
cost of SS-DGC is $\Or(N^3)$ and the memory cost is $\Or(N^2)$.  This is
undesirable and can be improved as in the next section with a more
efficient implementation.

%

\section{A robust and efficient implementation of the spectrum sweeping
method}\label{sec:robustdoslowrank}

%
%

The SS-DGC method can give very accurate estimation of the
DOS.  However, it also has two notable disadvantages:

\begin{enumerate}
  \item The SS-DGC method requires a rough estimate of the random
    vectors $N_{v}$.  If the number of random vectors $N_{v}$ is less
    than the \REV{numerical} rank of $g_{\sigma}(tI-A)$, then the estimated DOS will has
    $\Or(1)$ error. 
  \item The SS-DGC method requires the formation of the $Z(t)$ matrix
    for each point $t$.  The computational cost scales as
    $\Or(N^2 N_t)$ and the memory cost is $\Or(N N_t)$.  This is
    expensive if the number of points $N_{t}$ is large. 
\end{enumerate}

In this section we provide a more robust and efficient implementation of
the spectrum sweeping (RESS) method  to overcome the two problems above.
The main idea of the RESS-DGC method is to have

\begin{enumerate}
  \item a hybrid strategy for robust estimation of the DOS in the case
    when the number of vectors $N_{v}$ is insufficient at some points
    with at least $\Or(1/\sqrt{N_{v}})$ accuracy.
  \item a consistent method for directly computing of the matrix
    $Z^{*}(t_{i})Z(t_i)$ for each point $t_i$ and avoiding
    the computation and storage of $Z(t_{i})$.
\end{enumerate}

\subsection{A robust and efficient implementation for estimating the
trace of a \REV{numerically} low rank
matrix}

Given a \REV{numerically} low rank matrix $P$, we may apply
Alg.~\ref{alg:ranlowrank} to compute its low rank approximation using a
random matrix of size $N\times N_{v}$.  Let us denote the residual by
\begin{equation}
  R := P - Z B Z^{*}.
\end{equation}
If $N_{v}$ is large enough, then $\norm{R}_{F}$ should be very close to
zero.  Otherwise, Hutchinson's method can be used to estimate the trace
of $R$ as the correction for the trace of $P$. According to
Theorem~\ref{thm:EVarTrace}, if $ZBZ^{*}$ is relatively a good
approximation to $P$, the variance for estimating $\Tr[R]$ can be
significantly reduced.



To do this, we use another set of random vectors $\wt{W}\in
\mathbb{C}^{N\times \wt{N}_{v}}$, and 
\begin{equation}
  \Tr[R] \approx \frac{1}{\wt{N}_{v}} \Tr[\wt{W}^{*} R \wt{W}]
  = \frac{1}{\wt{N}_{v}}\left(\Tr[\wt{W}^{*} P \wt{W}]  -
  \Tr[(\wt{W}^{*} Z) B (Z^{*} \wt{W})]\right).
  \label{eqn:hybridcorrection}
\end{equation}


Eq.~\eqref{eqn:hybridcorrection} still requires the storage of
$Z$.  As explained above, storing $Z$ can become expensive if we 
use the same set of random vectors $W$ and $\wt{W}$ but evaluate the
DOS at a large number of points $\{t_{i}\}$.  In order to reduce such cost, note
that
\[
Z^{*}Z = W^{*}P^{*} PZ = W^{*}(P^2 W).
\]
If we can compute both $PW$ and $P^2 W$, then $Z$ does
not need to be explicitly stored.  Instead we only need to store
\[
K_{W} = W^{*} (PW),\quad  K_{Z} = W^{*} (P^2 W).
\]
\REV{This observation is used} in section~\ref{subsec:robustdos}, where $PW$ and
$P^2W$ are computed separately with Chebyshev expansion.

Similarly, for the computation of the correction
term~\eqref{eqn:hybridcorrection}, we can directly compute the cross
term due to $\wt{W}$ as
\[
K_{C} = \wt{W}^{*} Z = \wt{W}^{*} (P W), \quad K_{\wt{W}} = \wt{W}^{*} (P\wt{W}).
\]

Alg.~\ref{alg:ranlowrank} involves the computation of the pseudo-inverse
of $K_{W}$, which is in practice computed by solving a generalized eigenvalue problem
\REV{using Alg.~\ref{alg:geneig}.}
\REV{
The hybrid strategy in Eq.~\eqref{eqn:hybridcorrection} is 
consistent with the generalized eigenvalue problem, in the sense that}    
\begin{equation}
  \mathbb{E}_{\wt{w}} [\wt{w}^* Z \wt{C}\wt{C}^{*} Z^{*} \wt{w}] =
  \Tr[Z \wt{C}\wt{C}^{*}
  Z^{*}] = \Tr[\wt{C}^{*}(Z^{*}Z) \wt{C}] = \Tr[\wt{\Xi}].
  \label{eqn:consistentlowrank}
\end{equation}
The last equality of Eq.~\eqref{eqn:consistentlowrank} follows from that
$\wt{\Xi},\wt{C}$ solve the generalized eigenvalue problem as in
Alg.~\ref{alg:geneig}.


In summary, a robust and efficient randomized method for estimating the
trace of a low rank matrix is given in Alg.~\ref{alg:robustrantrace}.

\begin{algorithm}
\begin{small}
\begin{center}
  \begin{minipage}{5in}
\begin{tabular}{p{0.5in}p{4.5in}}
{\bf Input}:  &  \begin{minipage}[t]{4.0in}
  Hermitian matrix $P\in\mathbb{C}^{N\times N}$;
  Number of randomized vectors $N_{v},\wt{N}_{v}$\\
\end{minipage} \\
{\bf Output}:  &  \begin{minipage}[t]{4.0in}
  Estimated $\Tr[P]$.
\end{minipage} 
\end{tabular}
\begin{algorithmic}[1]
  \STATE Generate random Gaussian matrices $W\in \mathbb{R}^{N\times
  N_{v}}$ and $\wt{W}\in \mathbb{C}^{N\times \wt{N}_{v}}$.
  \STATE Compute $K_{W}=W^{*} (P W),\quad K_{Z}=W^{*} (P^2 W)$.
  \STATE Compute $K_{C}=\wt{W}^{*} (P W), K_{\wt{W}}=\wt{W}^{*} (P\wt{W})$.
  \STATE Solve the generalized eigenvalue problem
  \[
  K_{Z} \wt{C} = K_{W} \wt{C} \REV{\wt{\Xi}},
  \]
  \REV{using Alg.~\ref{alg:geneig}}.
  \STATE Compute the trace 
  \[ 
  \Tr[P]\approx \Tr[\wt{\Xi}] + \frac{1}{\wt{N}_{v}} \left(
  \Tr[K_{\wt{W}}] - \Tr[K_{C}\wt{C}\wt{C}^{*}K_{C}^{*}]
  \right).
  \] 
\end{algorithmic}
\end{minipage}
\end{center}
\end{small}
\caption{Robust and efficient randomized method for estimating the trace
of a \REV{numerically} low rank matrix.}
\label{alg:robustrantrace}
\end{algorithm}



\subsection{Robust and efficient implementation of the spectrum sweeping
method}\label{subsec:robustdos}

%


In order to combine Alg.~\ref{alg:robustrantrace} and
Alg.~\ref{alg:rlrdos} to obtain a robust and efficient implementation of
the spectrum sweeping method, it is necessary to evaluate
\REV{$(g_{\sigma}(tI-A))^2 W$}, where
\[
\REV{(g_{\sigma}(tI-A))^2} \equiv \frac{1}{N^2 2\pi\sigma^2}
e^{-\frac{(tI-A)^2}{\sigma^2}}.
\]
\REV{In order to do this,} one can directly compute
\REV{$(g_{\sigma}(tI-A))^2$ using  an auxiliary Chebyshev expansion as
follows} 
\begin{equation}
  (g_{\sigma}(tI-A))^2 \approx \sum_{l=0}^{M} \nu_{l}(t) T_{l}(A).
  \label{eqn:g2cheb}
\end{equation}
Proposition~\ref{prop:aliasingremoval} states that if the expansion is
chosen carefully, the Chebyshev expansion~\eqref{eqn:gcheb}
and~\eqref{eqn:g2cheb} are fully consistent, \REV{i.e. if $g_{\sigma}$
is expanded by a Chebyshev polynomial expansion of degree
$M/2$ denoted by $g_{\sigma}^{M/2}$, we can expand
$(g_{\sigma})^2$ using a Chebyshev polynomial expansion of degree $M$
denoted by $\wt{g}_{\sigma}^{M}$. These two expansions are consistent in the
sense that $(g_{\sigma}^{M/2})^2 =
\wt{g}_{\sigma}^{M}$.}

\begin{prop}
  Let $M$ be an even integer, and $P$ is a matrix polynomial function of
  $A$
  \[
  P = \sum_{l=0}^{M/2} \mu_{l} T_{l}(A).
  \]
  Then 
  \begin{equation}
    P^2 = \sum_{l=0}^{M} \nu_{l} T_{l}(A),
    \label{eqn:polychebsquare}
  \end{equation}
  where
  \[
  \nu_{l} = \frac{2-\delta_{l0}}{\pi}
  \int_{-1}^{1}\frac{1}{\sqrt{1-x^2}} T_{l}(x) \left( \sum_{k=0}^{M/2}
  \mu_{k} T_{k}(x) \right)^2 \REV{\ud x}, \quad l=0,\ldots,M.
  \]
  \label{prop:aliasingremoval}
\end{prop}
\begin{proof}
  Using the definition of $P$,
  \[
  P^2 = \sum_{p,q=0}^{M/2} \mu_{p} \mu_{q} T_{p}(A) T_{q}(A).
  \]
  Since $0\le p+q\le M$, $P^2$ is a polynomial of $A$ up to degree
  $M$, and can be expanded using a Chebyshev polynomial of the
  form~\eqref{eqn:polychebsquare}.  
\end{proof}

%
%

\REV{The expansion coefficient
$\nu_{l}$'s can be obtained using Alg.~\ref{alg:dgcexpansion} with
numerical integration, and} we have a consistent and efficient
method for estimating the DOS.
\begin{theorem}
  Let $A\in \mathbb{C}^{N\times N}$ be a Hermitian matrix, 
  and $W\in\mathbb{R}^{N\times N_{v}}$ be a random Gaussian matrix. 
  For any $t\in (-1,1)$, let $g^{M/2}_{\sigma}(tI-A)$ be the
  $M/2$ degree Chebyshev expansion
  \begin{equation}
    g^{M/2}_{\sigma}(tI-A) = \sum_{l=0}^{M/2} \mu_{l}(t)
    T_{l}(A),
    \label{eqn:gchebM2}
  \end{equation}
  and the coefficients $\{\nu_{l}\}_{l=0}^{M}$ are defined according to
  Proposition~\ref{prop:aliasingremoval}.  
  Define $Z(t) = g^{M/2}_{\sigma}(tI-A) W$, and
  \begin{equation}
    K_{W}(t) = W^{*} Z(t), \quad K_{Z}(t) = W^{*}\left( \sum_{l=0}^{M}
    \nu_{l}(t) T_{l}(A) W\right).
    \label{eqn:KWKZ}
  \end{equation}
  Let $\wt{w}$ be a random Gaussian vector, and
  \begin{equation}
    K_{C}(t) = \wt{w}^{*} Z(t), \quad K_{\wt{W}}(t) = \wt{w}^{*}
    g^{M/2}_{\sigma}(tI-A) \wt{w}.
    \label{}
  \end{equation}
  Then 
  \begin{equation}
    \Tr\left[ g^{M/2}_{\sigma}(tI-A)
    \right] = \mathbb{E}_{\wt{w}} \left(
    K_{\wt{W}}(t) - K_{C}(t) \wt{C}(t)
    \wt{C}^{*}(t) K_{C}^{*}(t) \right)  +
    \REV{\Tr[\wt{\Xi}(t)]}.
    \label{eqn:phiestimate}
  \end{equation}
  Here  \REV{$\wt{C}(t),\wt{\Xi}(t)$ solves the generalized eigenvalue
  problem  
  \begin{equation}
    K_{Z}(t) \wt{C}(t) = K_{W}(t) \wt{C}(t) \REV{\wt{\Xi}(t)},
    \label{eqn:generalevt}
  \end{equation}
  using Alg.~\ref{alg:geneig}}.
  \label{thm:ressdgc}
\end{theorem}
\begin{proof}
  By Proposition~\ref{prop:aliasingremoval},
  $\left(g^{M/2}_{\sigma}(tI-A)\right)^2$ can be exactly computed
  using a Chebyshev of degree $M$ with coefficients
  $\{\nu_{l}\}_{l=0}^{M}$.  Therefore
  \[
  K_{Z}(t) = Z^{*}(t) Z(t) = W^{*}\left( \sum_{l=0}^{M} \nu_{l}(t) T_{l}(A) W\right).
  \]
  The consistency between Hutchinson's method and the estimation of $\Tr[\wt{\Xi}(t)]$
  follows from Eq.~\eqref{eqn:consistentlowrank}.
\end{proof}

Following Theorem~\ref{thm:ressdgc}, the RESS-DGC algorithm is given in
Alg.~\ref{alg:rlrdos2}. 
\REV{Compared to DGC and SS-DGC method, the RESS-DGC method introduces
an additional parameter $\wt{N}_{v}$. 
In the numerical experiments, in order to carry out a fair
in the sense the \textit{total}
number of random vectors used are the same, i.e. for RESS-DGC, we always choose
$N_{v}+\wt{N}_{v}$ to be the same number of random vectors
$N_{v}^{\mathrm{DGC}}=N_{v}^{\mathrm{SS-DGC}}$ used in DGC and SS-DGC,
respectively. For instance, when $N_{v}^{\mathrm{DGC}}$ is relatively
small, in RESS-DGC we can
choose $N_{v}=\wt{N}_{v}=\frac12 N_{v}^{\mathrm{DGC}}$, i.e. half of the random vectors are used for low rank approximation, and half for hybrid
correction. When we are certain that $N_{v}$ is large enough, we can
eliminate the usage of the hybrid correction and take
$N_{v}=N_{v}^{\mathrm{DGC}}$ and $\wt{N}_{v}=0$.
}

\begin{algorithm}  
\begin{small}
\begin{center}
  \begin{minipage}{5in}
\begin{tabular}{p{0.5in}p{4.5in}}
{\bf Input}:  &  \begin{minipage}[t]{4.0in}
  Hermitian matrix $A$ with eigenvalues between $(-1,1)$;\\
  A set of points $\{t_{i}\}_{i=1}^{N_{t}}$ at which the DOS is to be evaluated;\\
  Polynomial degree $M$;\\
  Smearing parameter $\sigma$;\\
  Number of random vectors $N_{v},\wt{N}_{v}$.
\end{minipage} \\
{\bf Output}:  &  \begin{minipage}[t]{4.0in}
  Approximate DOS $\{\wt{\phi}_{\sigma}(t_i)\}$.
\end{minipage} 
\end{tabular}
\begin{algorithmic}[1]
  \STATE Compute the coefficient $\{\mu_{l}(t_{i})\}_{l=0}^{\frac{M}{2}}$ for each
  $t_{i},i=1,\ldots,N_{t}$ using Alg.~\ref{alg:dgcexpansion} for
  Eq.~\eqref{eqn:gcheb}, and let $\mu_{l}(t_{i})=0,\quad
  l=\frac{M}{2}+1,\ldots, M$.
  \STATE Compute the coefficient $\{\nu_{l}(t_{i})\}_{l=0}^{M}$ for each
  $t_{i},i=1,\ldots,N_{t}$ using Alg.~\ref{alg:dgcexpansion} for
  Eq.~\eqref{eqn:g2cheb}.
  \STATE Generate random Gaussian matrices $W\in \mathbb{R}^{N\times
  N_{v}}$ and $\wt{W}\in \mathbb{R}^{N\times \wt{N}_{v}}$.
  \STATE Initialize the three term recurrence matrices 
  $V_{m},V_{p}\gets \mathbf{0}\in \mathbb{C}^{N\times N_{v}},V_{c}\gets
  W$; $\wt{V}_{m},\wt{V}_{p}\gets \mathbf{0}\in \mathbb{C}^{N\times
  \wt{N}_{v}},\wt{V}_{c}\gets \wt{W}$.
  \STATE Initialize matrices $K_{W}(t_{i}),K_{Z}(t_{i})\gets
  \mathbf{0}\in \mathbb{C}^{N_{v}\times N_{v}},\quad i=1,\ldots,N_{t}$;
  $K_{C}(t_{i})\gets \mathbf{0}\in \mathbb{C}^{\wt{N}_{v}\times
  N_{v}}$ and $K_{\wt{W}}\gets \mathbf{0}\in \mathbb{C}^{\wt{N}_{v}\times
  \wt{N}_{v}},\quad i=1,\ldots,N_{t}$.
  \FOR {$l=0,\ldots,M$}
  \STATE Compute $X_{W} \gets W^{*}V_{c}, \quad X_{C} \gets \wt{W}^{*}
  V_{c}, \quad X_{\wt{W}} \gets \wt{W}^{*} \wt{V}_{c}$.
  \FOR {$i=1,\ldots, N_{t}$}
  \STATE Accumulate $K_{W}(t_{i})\gets K_{W}(t_{i}) + \mu_{l}(t_{i})
  X_{W},\quad K_{Z}(t_{i})\gets K_{Z}(t_{i}) + \nu_{l}(t_{i})
  X_{W}$.
  \STATE Accumulate $K_{C}(t_{i})\gets K_{C}(t_{i}) +
  \mu_{l}(t_{i}) X_{C},\quad K_{\wt{W}}(t_{i})\gets K_{\wt{W}}(t_{i}) + 
  \mu_{l}(t_{i}) X_{\wt{W}}$.
  \ENDFOR
  \STATE $V_{p}\gets (2-\delta_{l0}) A V_{c} - V_{m}$.
  \STATE $V_{m}\gets V_{c}, V_{c}\gets V_{p}$.
  \STATE $\wt{V}_{p}\gets (2-\delta_{l0}) A \wt{V}_{c} - \wt{V}_{m}$.  
  \STATE $\wt{V}_{m}\gets \wt{V}_{c}, \wt{V}_{c}\gets \wt{V}_{p}$.
  \ENDFOR
  \FOR {$i=1,\ldots, N_{t}$}
  \STATE Compute $\wt{\phi}_{\sigma}(t_{i})\gets 
  \Tr\left[ g^{\frac{M}{2}}_{\sigma}(t_{i}I-A)
  \right]  \approx \frac{1}{\wt{N}_{v}} \Tr\left[
  K_{\wt{W}}(t_{i}) - K_{C}(t_{i}) \wt{C}(t_{i})
  \wt{C}^{*}(t_{i}) K_{C}^{*}(t_{i}) \right]  +
  \Tr[\wt{\Xi}(t_{i})]$, where
  \REV{$\wt{\Xi}(t_{i})$,$\wt{C}(t_{i})$ are computed
  from Alg.~\ref{alg:geneig}.}
%
  \ENDFOR
\end{algorithmic}
\end{minipage}
\end{center}
\end{small}
\caption{Robust and efficient spectrum sweeping with
Delta-Gauss-Chebyshev (RESS-DGC) method for estimating the DOS.}
\label{alg:rlrdos2}
\end{algorithm}

\subsection{Complexity}\label{subsec:resscomplexy}

Following the same setup in section~\ref{subsec:resscomplexy}, we assume
that a series of matrices are spectrally uniformly distributed. 
We assume the Chebyshev polynomial
degree $M\sim \Or(N)$, and $c_{\mathrm{matvec}}\sim \Or(N)$. 
 \REV{In the complexity
analysis below, we neglect any contribution on the order of $\log N$.}
We also
assume that $N_{v}$ is kept as a constant and is omitted in the
asymptotic complexity count with respect to $N$ and $N_{t}$.

In terms of the computational cost, the computational cost for applying the
Chebyshev polynomial to the random matrix $W$ is $c_{\mathrm{matvec}} M
N_{v}\sim \Or(N^2)$. The cost for updating $K_{W}(t_i)$ and
$K_{Z}(t_i)$ is $N_{t} M N_{v}^2\sim \Or(N N_{t})$.  The cost
for computing the DOS by trace estimation for each $t_{i}$ is $\Or(N_{t}
N_{v}^3) \sim \Or(N_{t})$.  So the total cost is $\Or(N N_{t}+N^2)$.

In terms of the memory cost, the advantage of using
Alg.~\ref{alg:rlrdos2} also becomes clear here. The matrices $\{Z(t_{i})\}$
do not need to be computed or stored. Using the three-term recurrence
for Chebyshev polynomials, the matrices
$K_{W}(t),K_{Z}(t),K_{C}(t),K_{\wt{W}}(t)$ can be 
updated gradually, and the cost for storing these matrices are
$\Or(N_{v}^2 N_{t} \sim N_{t})$.  The cost for storing the matrix $W$ is
$\Or(N N_{v})\sim \Or(N)$.   So the total storage cost is
$\Or(N+N_{t})$.  

If we also assume $N_{t}\sim \Or(N)$ due to the vanishing choice of
$\sigma$, then the computational cost scales as $\Or(N^2)$ and the
storage cost scales as $\Or(N)$.

\section{Application to trace estimation of general matrix functions}\label{sec:traceest}

As an application for the accurate estimation of the
DOS, we consider the problem of estimating the trace of a smooth
matrix function as in Eq.~\eqref{eqn:tracefA}. 
In general $f(t)$ is not a localized function on
the real axis, and Alg.~\ref{alg:robustrantrace} based on low rank
decomposition cannot be directly used to estimate $\Tr[f(A)]$. 

However, if we assume that there exists $\sigma>0$ so that the Fourier
transform of $f(t)$ decays faster than the Fourier transform of
$g_{\sigma}(t)$, where $g_{\sigma}$ is a Gaussian function, \REV{then we
can} find a smooth function $\wt{f}(t)$ satisfying
\begin{equation}
  (\wt{f}*g_{\sigma})(t) = \int_{-\infty}^{\infty} \wt{f}(s)
  g_{\sigma}(t-s) \ud s = f(t).
  \label{eqn:deconvolv}
\end{equation}
The function $\wt{f}(t)$ can be obtained via a deconvolution procedure.
Formally
\begin{equation}
  \begin{split}
  \frac{1}{N}\Tr[f(A)] = & \int_{-\infty}^{\infty} f(t) \phi(t) \ud t 
  = \int_{-\infty}^{\infty} \int_{-\infty}^{\infty} \wt{f}(s)
  g_{\sigma}(t-s) \phi(t) \ud t \ud s \\
  = &\int_{-\infty}^{\infty} \wt{f}(s) \phi_{\sigma}(s) \ud s.
  \end{split}
  \label{eqn:traceconv}
\end{equation}
Eq.~\eqref{eqn:traceconv} states that the trace of the matrix function
$f(A)$ can be accurately computed from the integral of
$\wt{f}(s)\phi_{\sigma}(s)$, which is a now smooth function.  \REV{Since
the spectrum of $A$ is \REV{assumed to be} in the interval $(-1,1)$, the integration
range in Eq.~\eqref{eqn:traceconv} can be replaced by a finite
interval.}
The integral can be evaluated accurately via a trapezoidal rule, and
we only need the value of the DOS $\phi_{\sigma}$ evaluated on the
points requested by the quadrature. In such a way, we \REV{``transfer''} the
smoothness of $f(t)$ to the regularized DOS without losing accuracy.

The deconvolution procedure~\eqref{eqn:deconvolv} can be performed via a
Fourier transform. Assume that the eigenvalues of $A$ are further contained in
the interval $(-a,a)\subset (-1,1)$ \REV{with $0<a<1$}.
Then the Fourier transform requires that the function
$f(t)$ is a periodic function on a interval containing $(-a,a)$, which
is in general not satisfied in practice. However, note that the interval
in Eq.~\eqref{eqn:tracefA} does not require the exact function
$f(t)$. In fact
\[
\Tr[f(A)] = \int_{-a}^{a}
f(t)\phi(t) \ud t = \int_{-a}^{a} h(t) \phi(t) \ud t
\]
for any smooth function $h(t)$ satisfying
\begin{equation}
  h(t) = f(t), \quad t\in (-a,a).
  \label{eqn:hconstraint}
\end{equation}
In particular, $h(t)$ can be \REV{extended} to be a periodic function on the
interval $[-1,1]$.  In this work, we construct $h(t)$ as follows.
\begin{equation}
  h(t) = f(t) \pi(t) + \frac{f(-1)+f(1)}{2} \left( 1-\pi(t) \right).
  \label{eqn:hform}
\end{equation}
\REV{We remark that the constant $\frac{f(-1)+f(1)}{2}$ in
Eq.~\eqref{eqn:hform} is not important and can be in principle any
real number.}
Here $\pi(t)$ is a function with $\pi(t)=1$ for $t\in (-a,a)$, and
smoothly goes to $0$ outside $(-a,a)$.  It is easy to verify that such choice
of $h(t)$ satisfies Eq.~\eqref{eqn:hconstraint} and is periodic on
$(-1,1)$.  There are many choice of $\pi(t)$, and here we use
\begin{equation}
  \pi(t)  = \frac{1}{2}\left[ \mathrm{erf}\left(
  \frac{1+a-2t}{\wt{\sigma}} \right)-
  \mathrm{erf}\left(
  \frac{-1-a-2t}{\wt{\sigma}} \right) 
  \right].
  \label{eqn:pifunc}
\end{equation}
Here $\mathrm{erf}$ is the error function. 
\REV{If $\wt{\sigma}$ is chosen to be small enough, then} $\pi(t)$ as defined in
Eq.~\eqref{eqn:pifunc} satisfies the requirement.

Alg.~\ref{alg:resstrace} describes the procedure for computing the trace
of a matrix function.

\begin{algorithm}[h]
\begin{small}
\begin{center}
  \begin{minipage}{5in}
\begin{tabular}{p{0.5in}p{4.5in}}
{\bf Input}:  &  \begin{minipage}[t]{4.0in}
  Hermitian matrix $A$ with eigenvalues between $(-a,a)$ where
  $0<a<1$;\\ 
  Smooth function $f(t)$;\\
  Smearing parameter $\sigma,\wt{\sigma}$.
\end{minipage} \\
{\bf Output}:  &  \begin{minipage}[t]{4.0in}
  Estimated value of $\Tr[f(A)]$.
\end{minipage} 
\end{tabular}
\begin{algorithmic}[1]
  \STATE Compute auxiliary function $h(t)$ according to
  Eq.~\eqref{eqn:hform}.
\STATE Compute $\{\wt{f}(t_{i})\}_{i=1}^{N_{t}}$ satisfying $(\wt{f}*g_{\sigma})(t)=h(t)$
  on $(-1,1)$ through the Fourier transform on a uniform set of points
  $\{t_{i}\}_{i=1}^{N_{t}}$ with spacing $\Delta t=t_2-t_1$.
  \STATE \REV{Use one of the algorithms} to compute 
  $\{\wt{\phi}_{\sigma}(t_{i})\}$.
  \STATE Compute $\Tr[f(A)]\approx N \Delta t \sum_{i=1}^{N_{t}} \wt{f}(t_{i})
  \wt{\phi}_{\sigma}(t_{i})$.
\end{algorithmic}
\end{minipage}
\end{center}
\end{small}
\caption{Spectrum sweeping method for estimating the trace of a matrix function.}
\label{alg:resstrace}
\end{algorithm}

\FloatBarrier

\section{Numerical examples}\label{sec:numer}

In this section we demonstrate the accuracy and efficiency of the SS-DGC
and RESS-DGC methods for computing the spectral densities and for estimating
the trace.  \REV{For the asymptotic scaling of the
method, we need a series of matrices that are approximately spectrally
uniformly distributed. These are given by the ModES3D\_X matrices to be
detailed below. In order to demonstrate that the methods are also
applicable to general matrices, we also give test results for two
matrices obtained from the University of Florida matrix
collection~\cite{FloridaMatrix}. }
All
the computation is performed on a single computational thread of an
Intel i7 CPU processor with $64$ gigabytes (GB) of memory using MATLAB.

As a
model problem, we consider a second order partial differential operator
$\hat{A}$ in a three-dimensional (3D) cubic domain with periodic
boundary conditions.  For a smooth function $u(x,y,z)$, $\hat{A}u$ is
given by
\[
(\hat{A} u)(x,y,z) = -\Delta u(x,y,z) + V(x,y,z) u(x,y,z).
\]
The matrix $A$ is obtained from a $7$-point finite difference
discretization of $\hat{A}$.  In order to create a series of matrices,
first we consider one cubic domain $\Omega=[0,L]^3$ and $V(x,y,z)$ is
taken to be a Gaussian function centered at
$(L/2,L/2,L/2)^{T}$. This is called a ``unit cell''. The unit cell is then
extended by $n$ times along the $x,y,z$ directions, respectively, and
the resulting domain is $[0,nL]^3$ and $V(x,y,z)$ is the linear
combination of $n^3$ Gaussian functions.  Such matrix can be interpreted as a
model matrix for electronic structure calculation.  Each
dimension of the domain is uniformly discretized with grid spacing $h$,
and the resulting matrix $A$ is denoted by ModES3D\_X where X is the
total number of unit cells.  Here we take $L=6$ and $h=0.6$. Some
characteristics of the matrices, including the dimension, the smallest
and the largest eigenvalue are given in Table~\ref{tab:testmat}.
Fig.~\ref{fig:Vpot} (a), (b) shows the isosurface of one example of such
potential for the matrix ModES3D\_1, and ModES3D\_8, respectively.
Fig.~\ref{fig:ModDOS} shows the DOS corresponding to low lying
eigenvalues for the matrices with $X=1,8,27,64$ for a fixed
regularization parameter $\sigma=0.02$, which indicate that the spectral
densities is roughly uniform.  

\begin{table}[h]
  \centering
  \REV{
  \begin{tabular}{c|c|c|c}
    \hline
    Matrix & $N$ & min(ev) & max(ev)  \\ 
    \hline
    ModES3D\_1 & 1000 & -2.22 & 32.23 \\
    ModES3D\_8 & 8000 & -2.71 & 31.31 \\
    ModES3D\_27 & 27000 & -2.75 & 31.30 \\
    ModES3D\_64 & 64000 & -2.76 & 32.30 \\
    \hline
    pe3k & 9000 & $8\times 10^{-6}$ & 127.60 \\
    shwater & 81920 & 5.79 & 20.30 \\
    \hline
  \end{tabular}
  }
  \caption{Characteristics of the test matrices.}
  \label{tab:testmat}
\end{table}

\begin{figure}[h]
  \begin{center}
    \subfloat[(a) ModES3D\_1]{\includegraphics[width=0.35\textwidth]{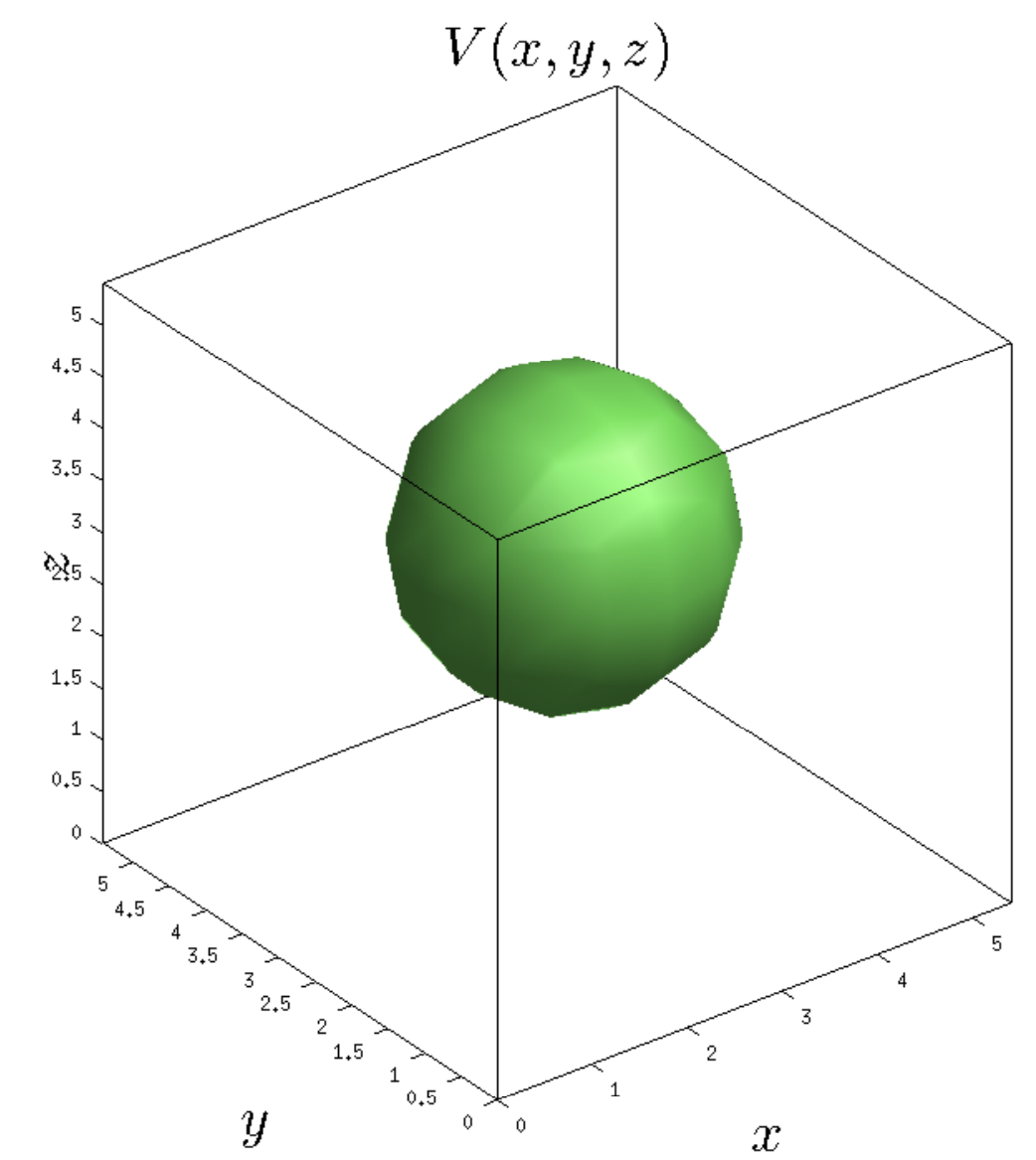}} \quad
    \subfloat[(b) ModES3D\_8]{\includegraphics[width=0.35\textwidth]{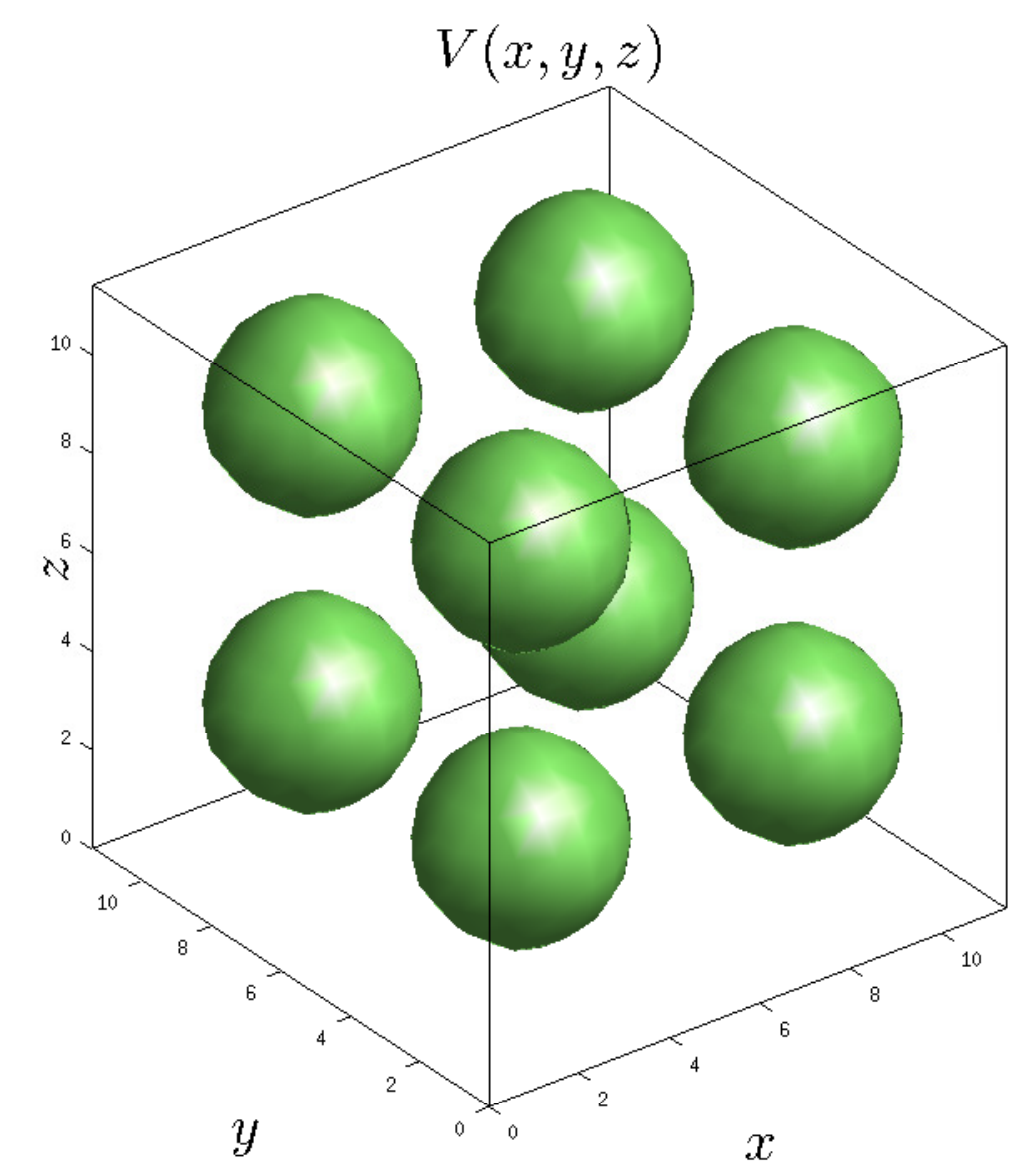}}
  \end{center}
  \caption{Isosurface for $V(x,y,z)$.}
  \label{fig:Vpot}
\end{figure}

\begin{figure}[h]
  \begin{center}
    \includegraphics[width=0.40\textwidth]{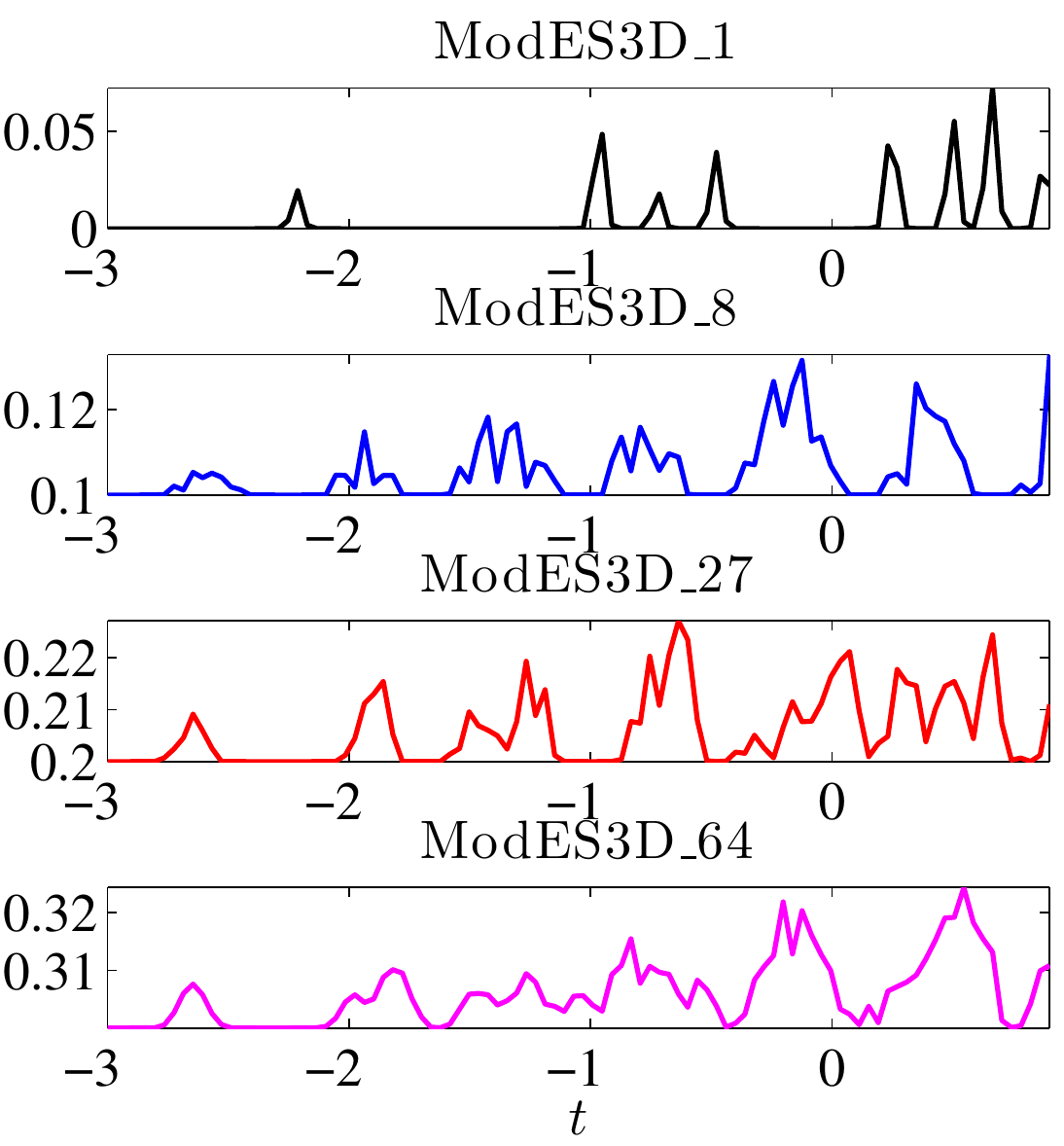}
  \end{center}
  \caption{Shape of the DOS for a series of matrices with
  $\sigma=0.02$.}
  \label{fig:ModDOS}
\end{figure}

\subsection{Spectral densities}\label{subsec:specden}

In order to compare with the accuracy of the DOS, the exact DOS is
obtained by solving eigenvalues corresponding to the region of interest
for computing the DOS. We use the locally optimal block preconditioned
conjugate gradient (LOBPCG)~\cite{Knyazev2001} method. The LOBPCG method
is advantageous for solving a large number of eigenvalues and
eigenvectors since it can effectively take advantage of the BLAS3
operations by solving all eigenvectors simultaneously. The number of
eigenvectors to be computed is set to be $35 X$ for the test
matrices ModES3D\_X with $X=1,8,27,64$, respectively, and the highest
eigenvalue obtained with such number of eigenvectors is slightly larger
than $1.0$.  The tolerance of LOBPCG is set to be $10^{-8}$ and the
maximum number of iterations is set to be $1000$. 

We measure the error of the approximate DOS using the relative
$L^{1}$ error defined as follows. Denote by
$\wt{\phi}_{\sigma}(t_{i})$ 
the approximate DOS evaluated at a series of uniformly distributed
points $t_{i}$, and by $\phi_{\sigma}(t_{i})$ the exact DOS obtained from
the eigenvalues.  Then the error is defined as 
\begin{equation}
  \mathrm{Error} = \frac{\sum_{i}
  \abs{\wt{\phi}_{\sigma}(t_{i})-\phi_{\sigma}(t_{i})}}{\sum_{i}
  \abs{\phi_{\sigma}(t_{i})}}.
  \label{eqn:error}
\end{equation}

%
For the DGC and SS-DGC method, an $M$-th order Chebyshev polynomial is
used to evaluate $Z(t_{i})$. For the RESS-DGC method, an $M$-th order
Chebyshev polynomial should be used to expand $Z^{*}(t_{i})Z(t_i)$.
Following Theorem~\ref{thm:ressdgc}, only an $M/2$-th order Chebyshev
polynomial can be used to evaluate $Z(t_i)$.  Similarly when
$N_{v}$ random vectors are used for DGC and SS-DGC, RESS-DGC is a hybrid
method containing two terms.  
\REV{
As discussed in section~\ref{subsec:robustdos}, the number of random
vectors used in DGC and SS-DGC are the same i.e.
$N_{v}^{\mathrm{DGC}}=N_{v}^{\mathrm{SS-DGC}}$.
For RESS-DGC, we use $N_{v}^{\mathrm{RESS-DG}}=\frac12
N_{v}^{\mathrm{DGC}}$ for the low rank approximation, and
$\wt{N}_{v}^{\mathrm{RESS-DGC}}=\frac12 N_{v}^{\mathrm{DGC}}$ for the
hybrid correction.}
This setup makes sure
that all methods perform \textit{exactly the same} number of matrix-vector
multiplications. 


Fig.~\ref{fig:lap3dnumUnit2Mdeg} (a) shows the error of the three methods
for the ModES3D\_8 matrix when a very small number of random vectors
$N_{v}=40$ is used, with increasing polynomial degrees $M$ from $200$ to
$3200$. Here we use $\sigma=0.05$.
Note the relatively large polynomial degree is mainly due to the
relatively large spectral radius compared to the desired resolution as
in Table~\ref{tab:testmat}. This can be typical in practical
applications.
DGC is slightly more accurate for
low degree of polynomials, but as $M$ increases, RESS-DGC becomes more
accurate.  
Fig.~\ref{fig:lap3dnumUnit2Mdeg} (b) shows the error when a relatively
large number of random vectors $N_{v}=160$ is used. When $M$ is large
enough, both SS-DGC and RESS-DGC can be significantly more accurate than
DGC. SS-DGC is slightly more accurate here, because it uses the
Chebyshev polynomials and random vectors more optimally than RESS-DGC,
though the computational cost can be higher when spectral densities at a
large number of points $N_{t}$ need to be evaluated.
%

\begin{figure}[h]
  \begin{center}
    \subfloat[(a)]{\includegraphics[width=0.35\textwidth]{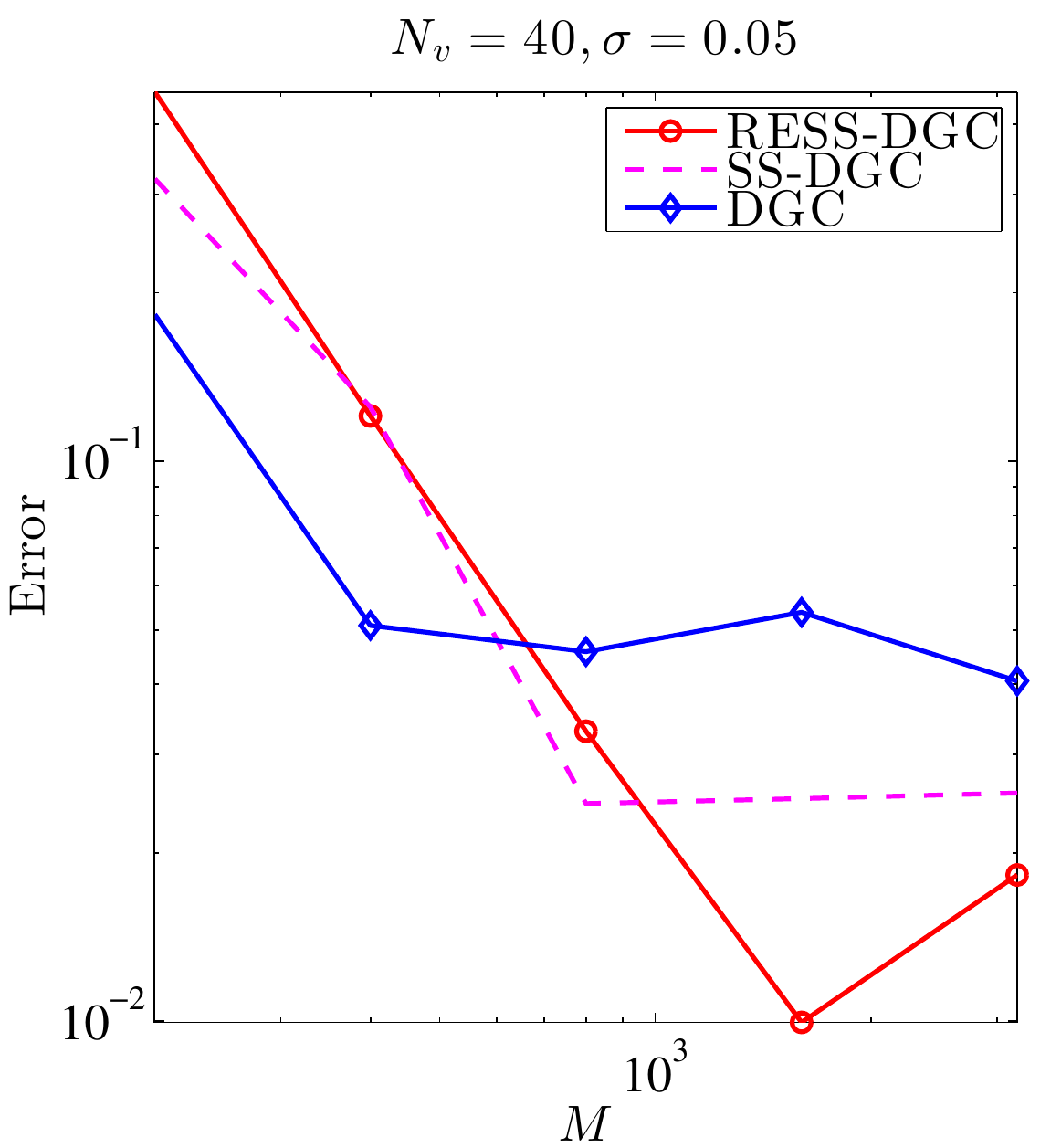}}
    \subfloat[(b)]{\includegraphics[width=0.35\textwidth]{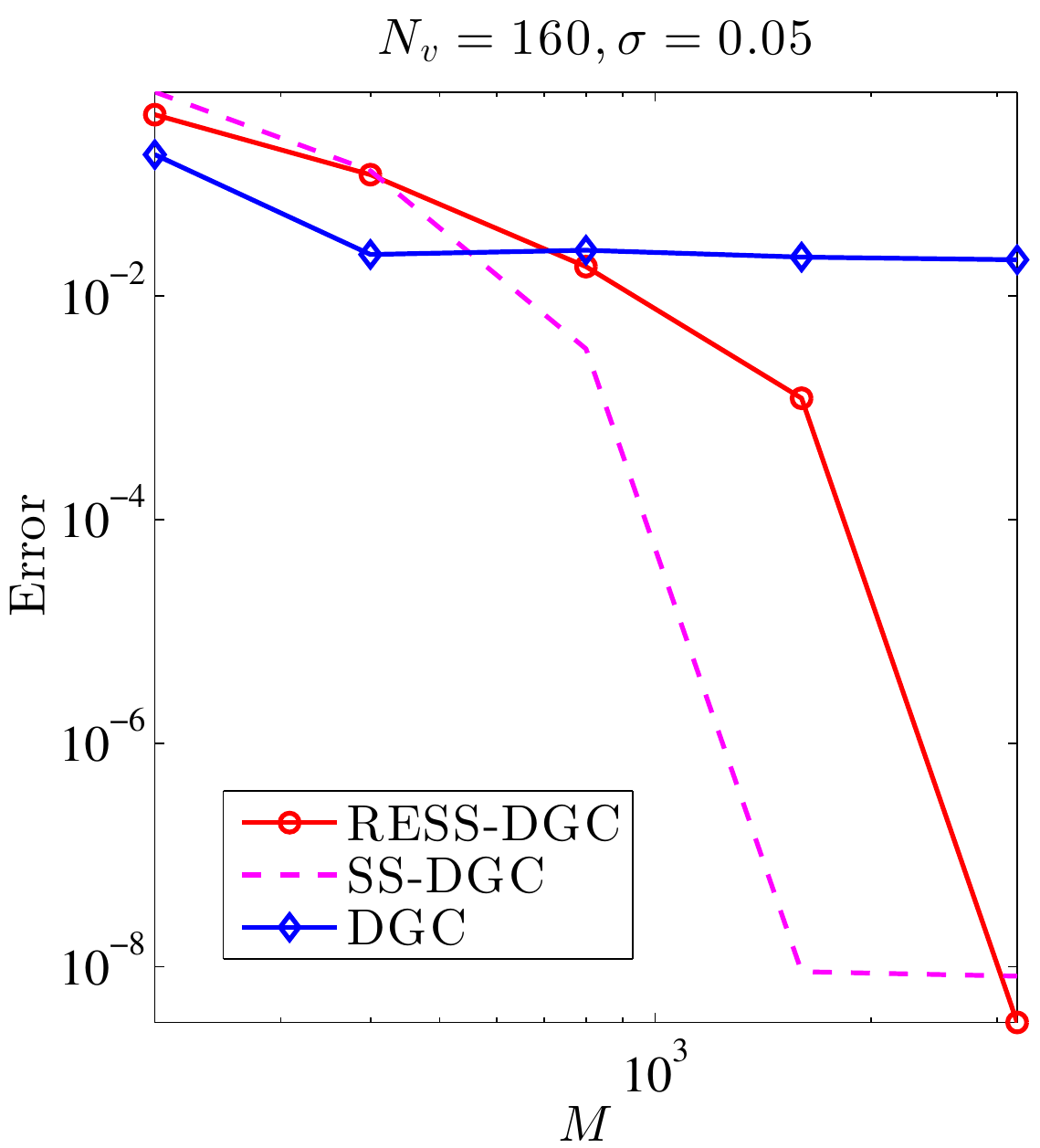}}
  \end{center}
  \caption{For the ModES3D\_8 matrix, the error of the DOS with respect
  to increasing polynomial degrees for (a) A small number of random
  vectors $N_{v}=40$ (b) A large number of random vectors $N_{v}=160$.}
  \label{fig:lap3dnumUnit2Mdeg}
\end{figure}

Fig.~\ref{fig:lap3dnumUnit2Nvec} (a) shows the comparison of the
accuracy of three methods for a relatively low degree of polynomials
$M=800$ and with increasing number of random vectors $N_{v}$.  When
$N_{v}$ is small, SS-DGC has $\Or(1)$ error, and this error is much
suppressed in RESS-DGC thanks to the hybrid correction scheme.  It is
interesting to see that RESS-DGC outperforms DGC for all
choices of $N_{v}$, but its accuracy is eventually limited by the
insufficient number of Chebyshev polynomials to expand the Gaussian
function.  SS-DGC is more accurate than RESS-DGC when $N_{v}$ is large
enough. This is because in such case the low rank decomposition captures the correlation between the results obtained among
different random vectors more
efficiently. On the contrary,  Hutchinson's method for
which DGC relies on only reduces the error only through direct Monte
Carlo sampling.  Fig.~\ref{fig:lap3dnumUnit2Nvec} (b) shows the case
when a relatively large number of polynomials $M=2400$ is used. Again
for small $N_{v}$, RESS-DGC \REV{reduces} the large error \REV{compared
to} the SS-DGC
method, while for large enough $N_{v}$ both SS-DGC and RESS-DGC can be
very accurate.

\begin{figure}[h]
  \begin{center}
    \subfloat[(a)]{\includegraphics[width=0.35\textwidth]{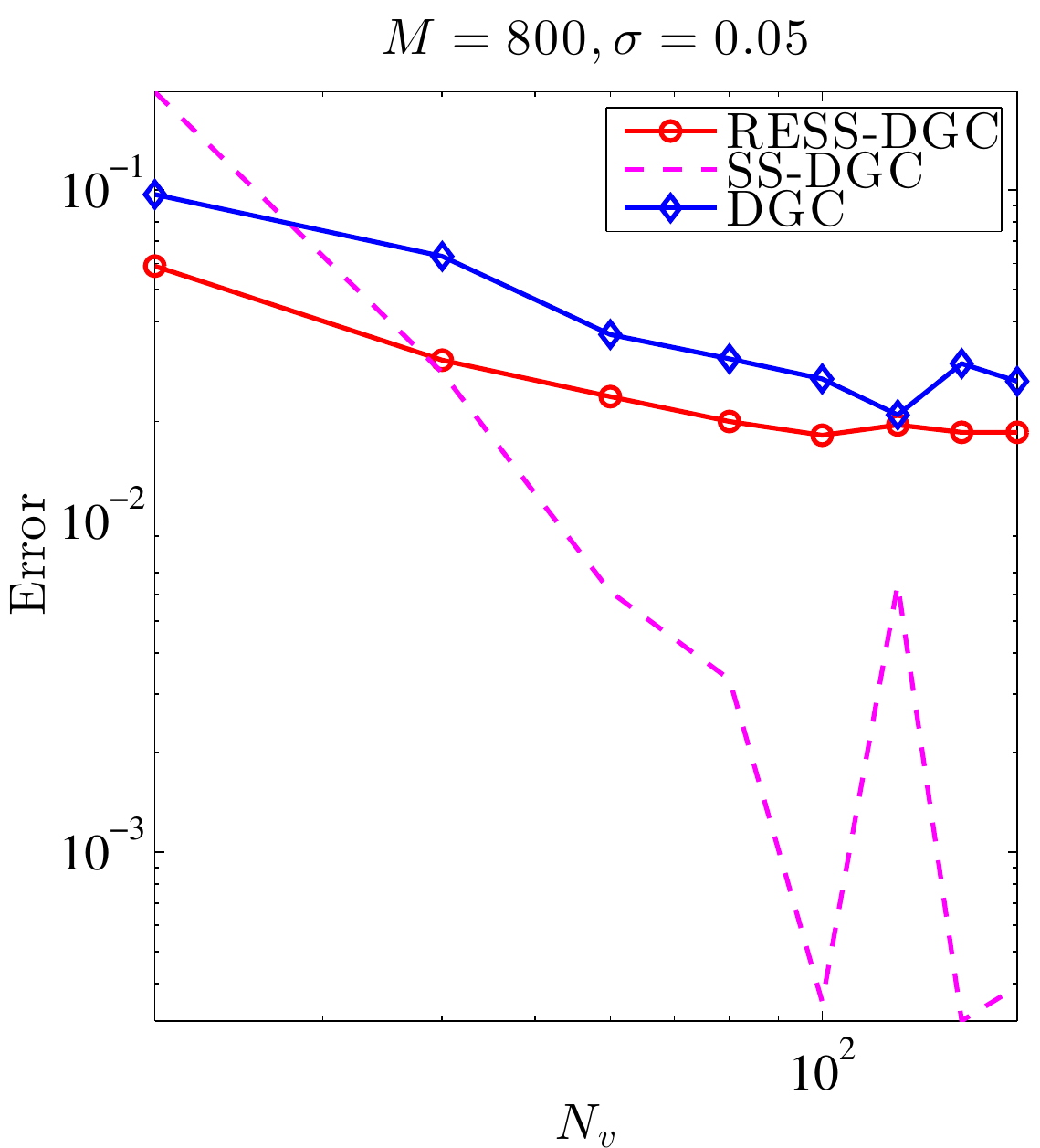}}
    \subfloat[(b)]{\includegraphics[width=0.35\textwidth]{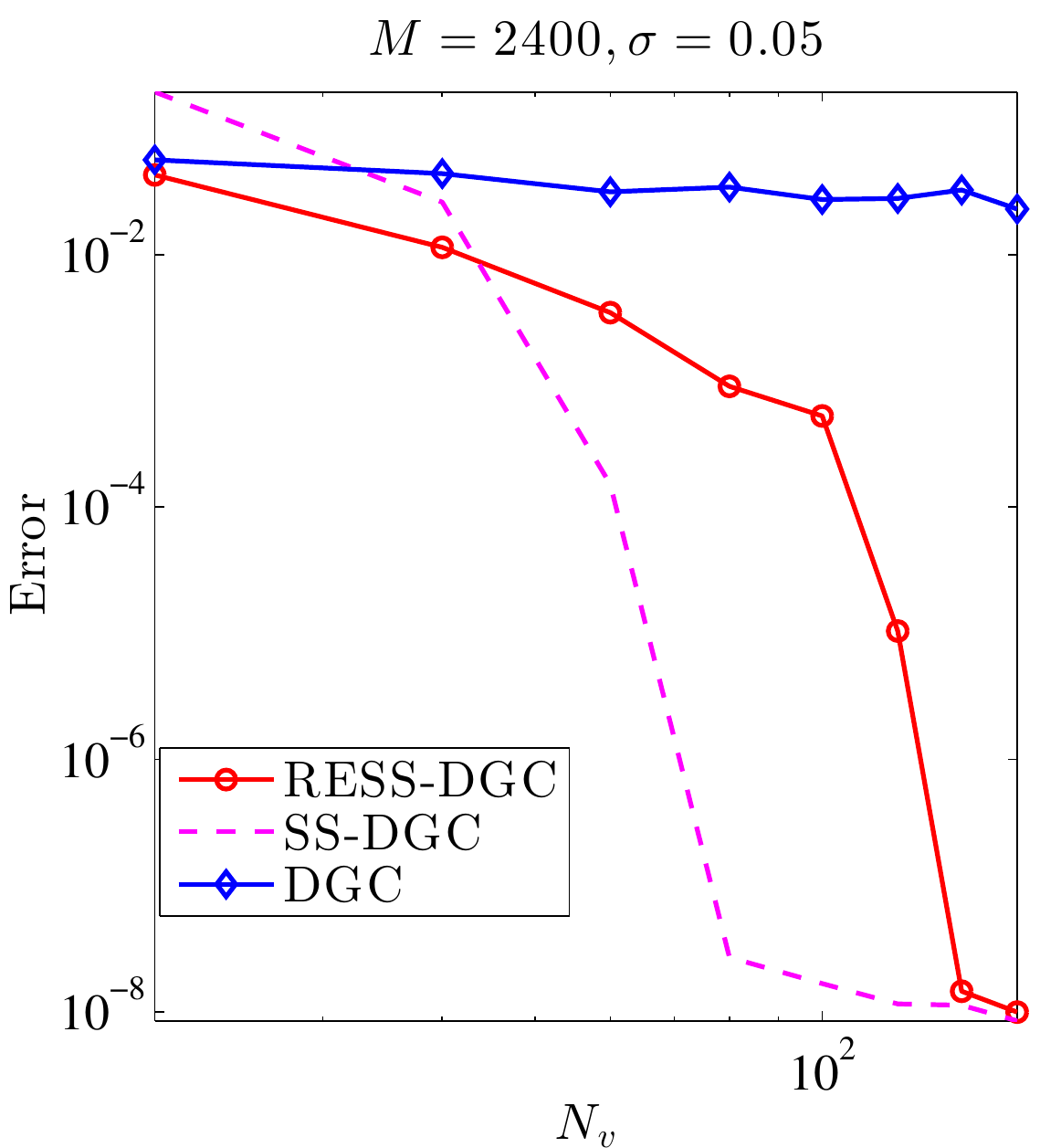}}
  \end{center}
  \caption{For the ModES3D\_8 matrix, the error of the DOS with respect
  to increasing random vectors for (a) A low polynomial degree
  $M=800$ (b) A high polynomial degree $M=2400$.}
  \label{fig:lap3dnumUnit2Nvec}
\end{figure}

\REV{In both SS-DGC and RESS-DGC methods, the parameter $\sigma$ is
important since it determines both the degrees of Chebyshev polynomial
to accurately expand $g_{\sigma}$, and the number of random
vectors needed for accurate low rank approximation. Fig.~\ref{fig:lap3dnumUnit2sigma} shows the error of
DGC, SS-DGC and RESS-DGC with $\sigma$ varying from $0.02$ to $0.1$.
Here we choose $M=120/\sigma$ and $N_{v}=3200\sigma$. This
corresponds to the case when a relatively high degrees of Chebyshev
polynomial and a relatively large number of random vectors are used in
the previous discussion. We observe that the scaling $M\sim
\Or(\sigma^{-1})$ and $N_{v}\sim \Or(\sigma)$ is important for spectrum
sweeping type methods to be accurate, and SS-DGC and RESS-DGC can
significantly outperform DGC type methods in terms of
accuracy.
}

\begin{figure}[h]
  \begin{center}
    \includegraphics[width=0.35\textwidth]{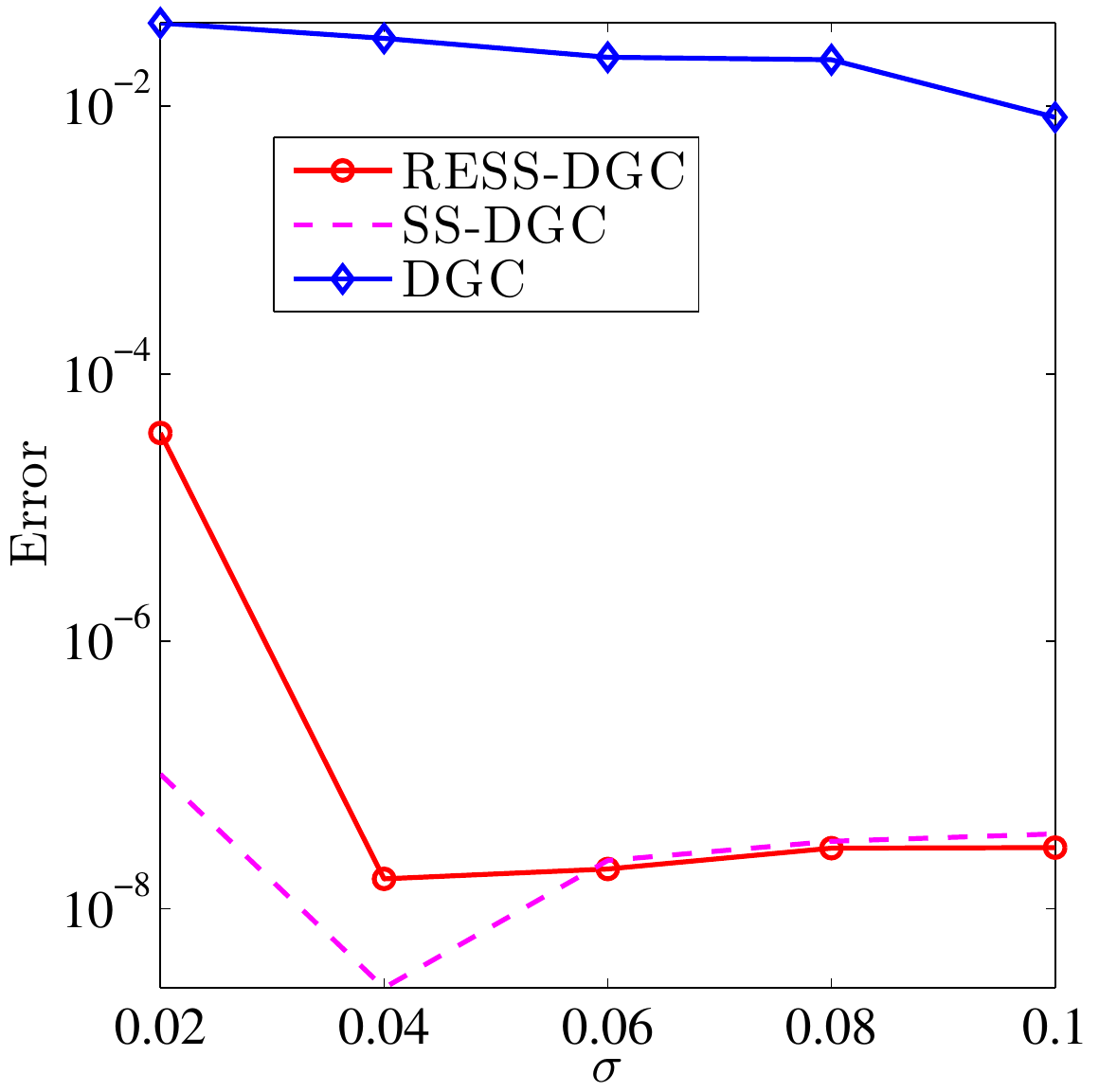}
  \end{center}
  \caption{\REV{For the ModES3D\_8 matrix, the error of the DOS with respect
  to different choices of $\sigma$.}}
  \label{fig:lap3dnumUnit2sigma}
\end{figure}

In order to study the weak scalability of the methods using the
ModES3D\_X matrices, as given in the complexity analysis, the polynomial
degrees $M$ should be chosen to be proportional to $X$. Here $M=300 X$
for $X=1,8,27,64$, respectively.  Correspondingly $\sigma=0.4/X$ and
$N_{t}=5 X$.  This allows us to use the same number of random vectors
$N_{v}=150$ for all matrices.  Fig.~\ref{fig:lap3dscalen} shows the
wall clock time of the three methods.   Both SS-DGC
and LOBPCG are asymptotically $\Or(N^3)$ methods with respect to the
increase of the matrix size, and the cubic scaling becomes apparent from $X=27$ to
$X=64$. RESS-DGC is only slightly more expensive than DGC and scales
as $\Or(N^2)$. For ModES3D\_64, the wall clock time for DGC, RESS-DGC,
SS-DGC and LOBPCG is 2535, 3293, 11979, 49389 seconds, respectively. 
Here RESS-DGC is $15$ times faster than LOBPCG and is the most
effective method.  Fig.~\ref{fig:lap3dscalen} (b) shows the accuracy in
terms of the relative $L^1$ error. Both SS-DGC and RESS-DGC can be
significantly more accurate than DGC.  We remark that since $N_{v}$ is
large enough in all cases here, the efficiency of RESS-DGC can be
further improved by noting that it only effectively uses half of the
random vectors here, due to the small contribution from the other half
of random vectors used for the hybrid correction.

\begin{figure}[h]
  \begin{center}
    \subfloat[(a)]{\includegraphics[width=0.4\textwidth]{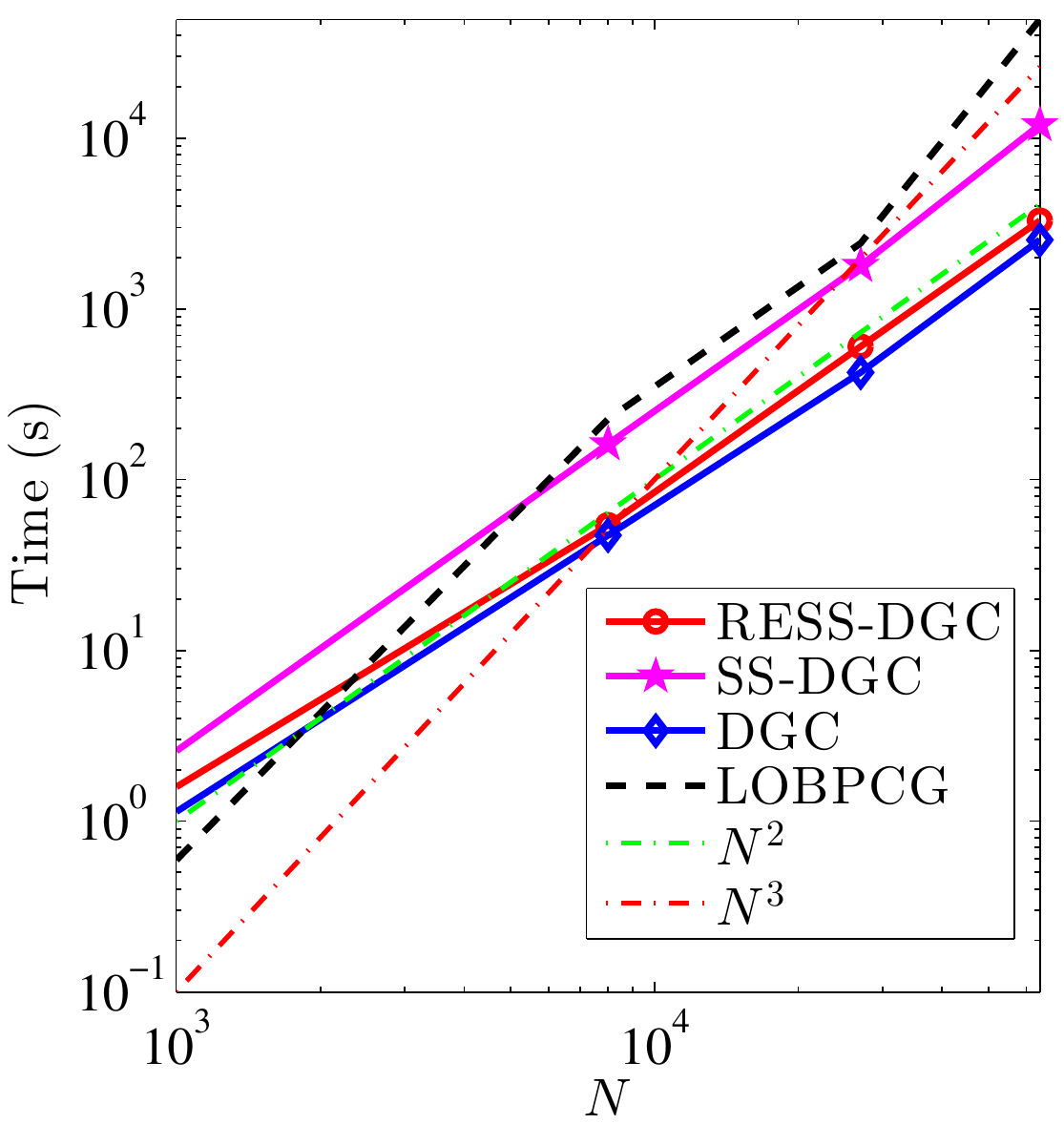}}
    \subfloat[(b)]{\includegraphics[width=0.4\textwidth]{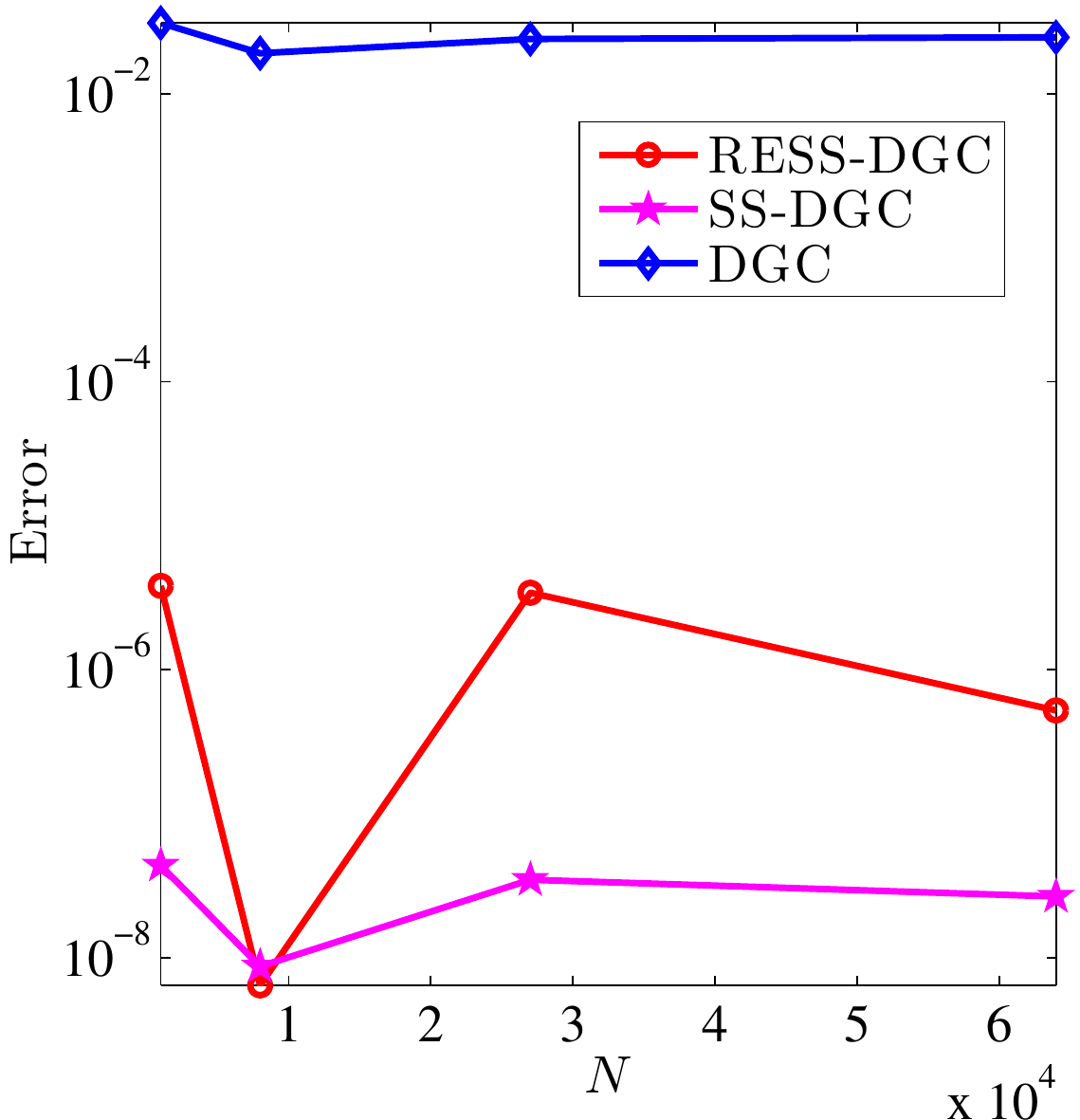}}
  \end{center}
  \caption{(a) Wall clock time of the DGC, SS-DGC and RESS-DGC methods compared with
  diagonalization method using LOBPCG. (b) The
  error of the DOS computed at RESS-DGC and the DGC method.}
  \label{fig:lap3dscalen}
\end{figure}

%



%




\subsection{Trace estimation}

As discussed in section~\ref{sec:traceest}, the accurate calculation of
the regularized DOS can be used for trace estimation. To demonstrate
this, we use the same ModES3D\_X matrices, and let 
$f(A)$ be the Fermi-Dirac function, i.e.
\[
\Tr[f(A)] =  \Tr\left[ \frac{1}{1+\exp(\beta(A-\mu I)} \right]
\]
is to be computed.  In electronic structure calculation, $\mu$ has the
physical meaning of chemical potential, and $\beta$ is the inverse
temperature.  The trace of the Fermi-Dirac distribution has the physical
meaning of the number of electrons at chemical potential $\mu$.  Here
$\beta=10.0,\mu=-1.0$. The value of $\sigma$ that can be used for the
deconvolution procedure in Eq.~\eqref{eqn:traceconv} should be chosen
such that after deconvolution $\wt{f}(s)$ is still a smooth function.
Here we use $\sigma=0.05$ for $X=1$, and  $\sigma=0.4/X$ for
$X=8,27,64$, respectively.  \REV{The value of $\wt{\sigma}$ for the
smearing function in Eq.~\eqref{eqn:pifunc} is $0.016$.} Correspondingly the polynomial degree $M=300
X$.  The number of random vectors $N_{v}$ is kept to be $100$ for all
calculations.

Fig.~\ref{fig:tracelap3d} shows the relative error of the trace for DGC,
SS-DGC and RESS-DGC. We observe that SS-DGC and RESS-DGC can be
significantly more accurate compared to DGC, due to the better
use of the correlated information obtained among different random
vectors.  Again when $N_{v}$ is sufficiently large, SS-DGC is more accurate since
the hybrid strategy in RESS-DGC is no longer needed here.

\begin{figure}[h]
  \begin{center}
    \includegraphics[width=0.4\textwidth]{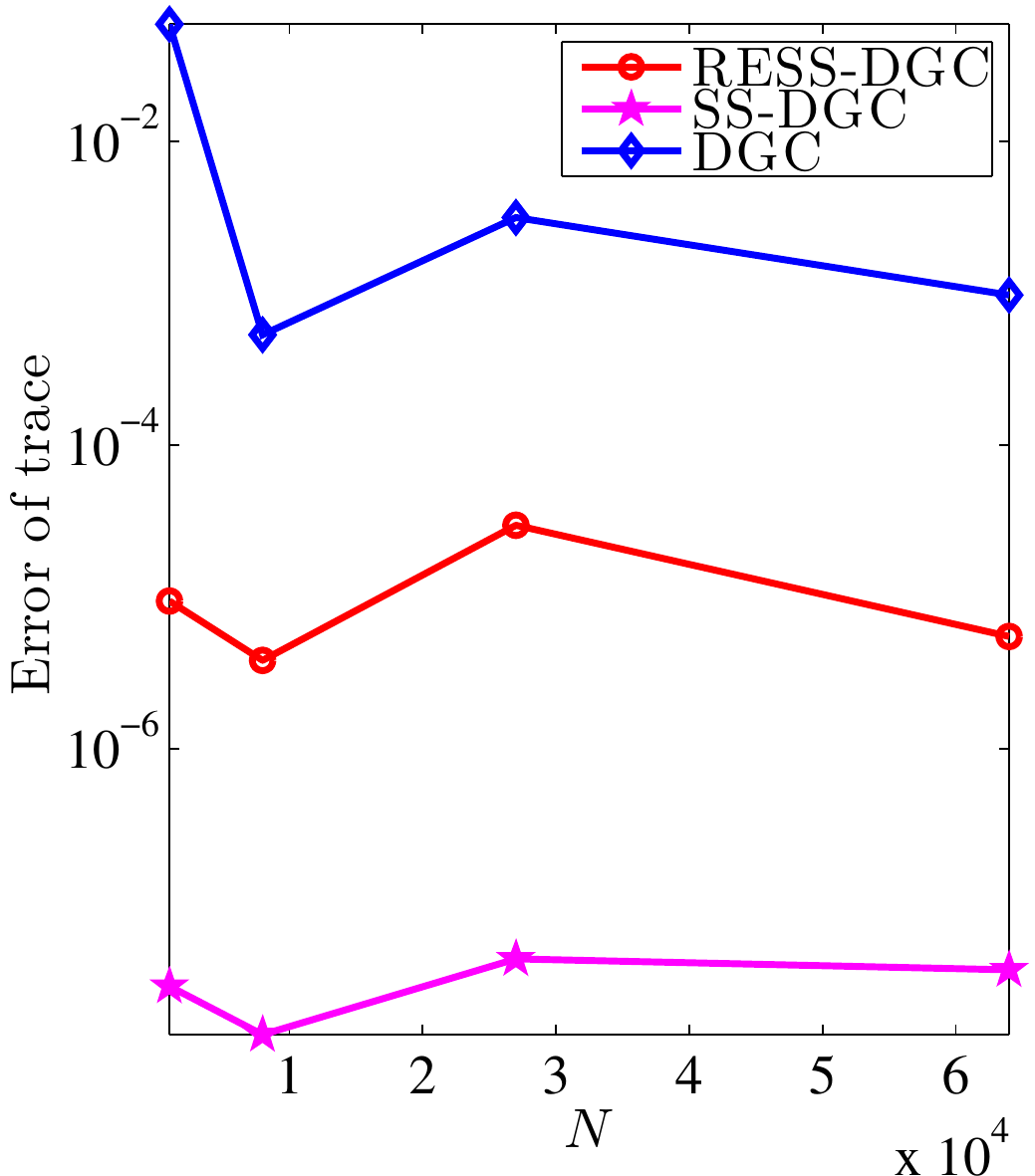} 
  \end{center}
  \caption{Relative error for estimating the trace of Fermi-Dirac
  functions applied to ModES3D\_X matrices.}
  \label{fig:tracelap3d}
\end{figure}

\subsection{Other matrices}
\REV{In section~\ref{subsec:specden}, we verified that both SS-DGC and
RESS-DGC can obtain very accurate estimation of the DOS when the
degrees of Chebyshev polynomial and the number of random vectors are
large enough, and RESS-DGC can lead to more efficient implementation.
In this section we further verify that RESS-DGC can achieve high
accuracy for other test matrices, when the degrees of polynomial and
number of vectors $N_{v}$ is large enough. The test matrices
pe3k and shwater are obtained from the University of Florida matrix
collection~\cite{FloridaMatrix}, and are used as test
matrices in~\cite{LinSaadYang2015}. The character of the matrices is 
given in Table~\ref{tab:testmat}. 
}

\REV{
Fig.~\ref{fig:pe3kerror} shows 
the DOS obtained from RESS-DGC for the pe3k matrix, compared to the DOS obtained by
diagonalizing the matrix directly (``Exact''). The parameters are 
$\sigma=0.25$,$M=4084$,$N_{v}=300$,$N_{t}=100$. Since the goal is to demonstrate high
accuracy, we turn off the hybrid mode in the RESS-DGC method by setting
$\wt{N}_{v}=0$. Fig.~\ref{fig:pe3kerror} shows that the error of
RESS-DGC is less than $10^{-7}$ everywhere. The relatively large error
occurs at the two peaks of the DOS, which agrees with the theoretical
estimate that RESS-DGC needs more random vectors when the spectral
density is large. Similarly Fig.~\ref{fig:shwatererror}  shows the same
comparison for the shwater matrix,
with $\sigma=0.005$,$M=16240$,$N_{v}=640$,$\wt{N}_{v}=0$, and $N_{t}=100$.
More detailed comparison of the error and running time for the two
matrices with increasing number of random vectors $N_{v}$ is given in 
Table~\ref{tab:errorgeneral} and Table~\ref{tab:timegeneral},
respectively. We observe that the error of RESS-DGC rapidly decreases
with respect to the increase of the number of random vectors, while the
error of DGC only decreases marginally. We find that RESS-DGC only
introduces marginally extra cost
compared to the cost of the DGC method.}  

\begin{figure}[h]
  \begin{center}
    \subfloat[(a)]{\includegraphics[width=0.3\textwidth]{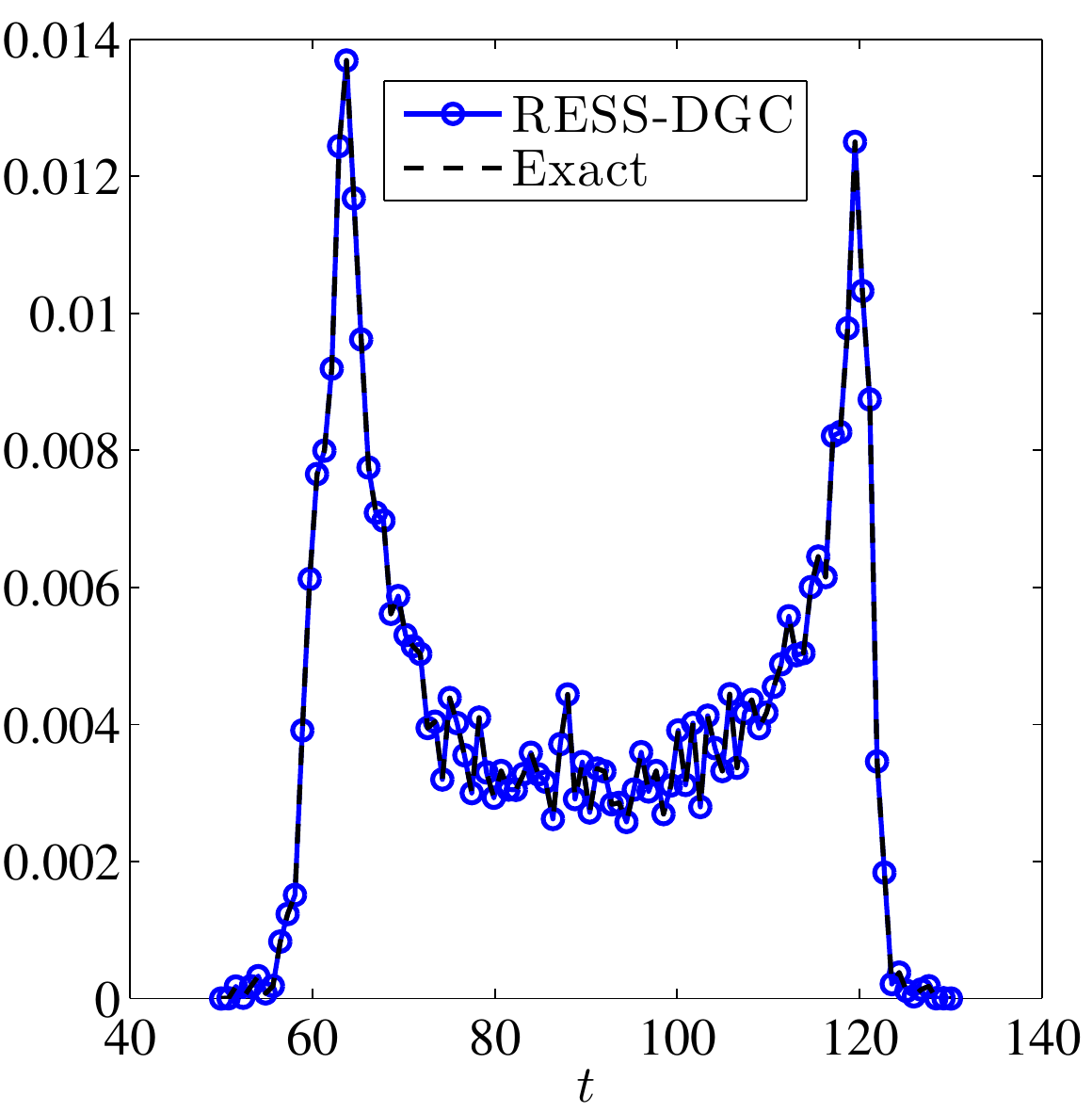}}
    \quad
    \subfloat[(b)]{\includegraphics[width=0.3\textwidth]{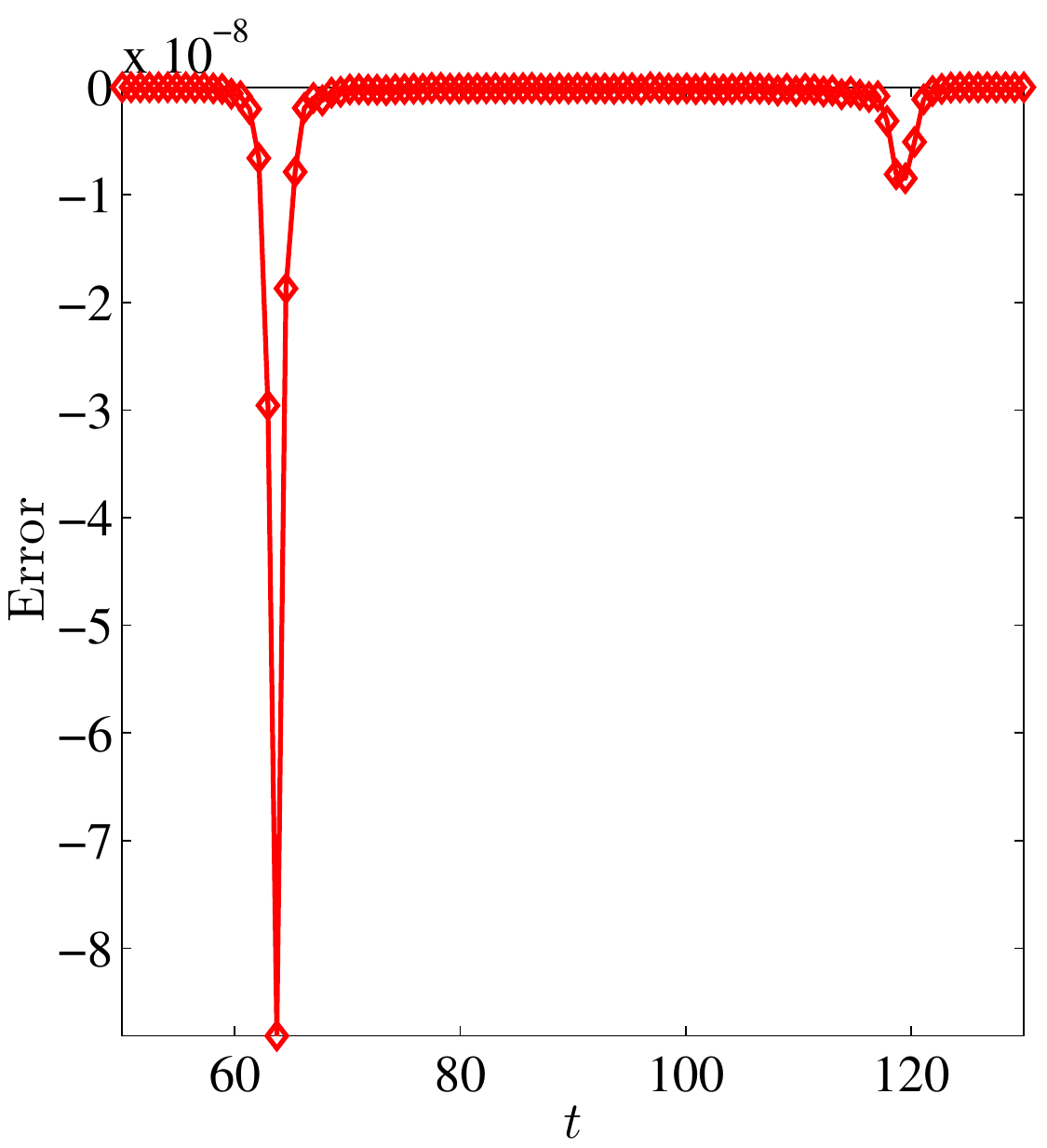}}
  \end{center}
  \caption{\REV{For the pe3k matrix, (a) numerically computed DOS by RESS-DGC, compared to
  the exact DOS, and (b) the error of the DOS computed by RESS-DGC.}}
  \label{fig:pe3kerror}
\end{figure}

\begin{figure}[h]
  \begin{center}
    \subfloat[(a)]{\includegraphics[width=0.3\textwidth]{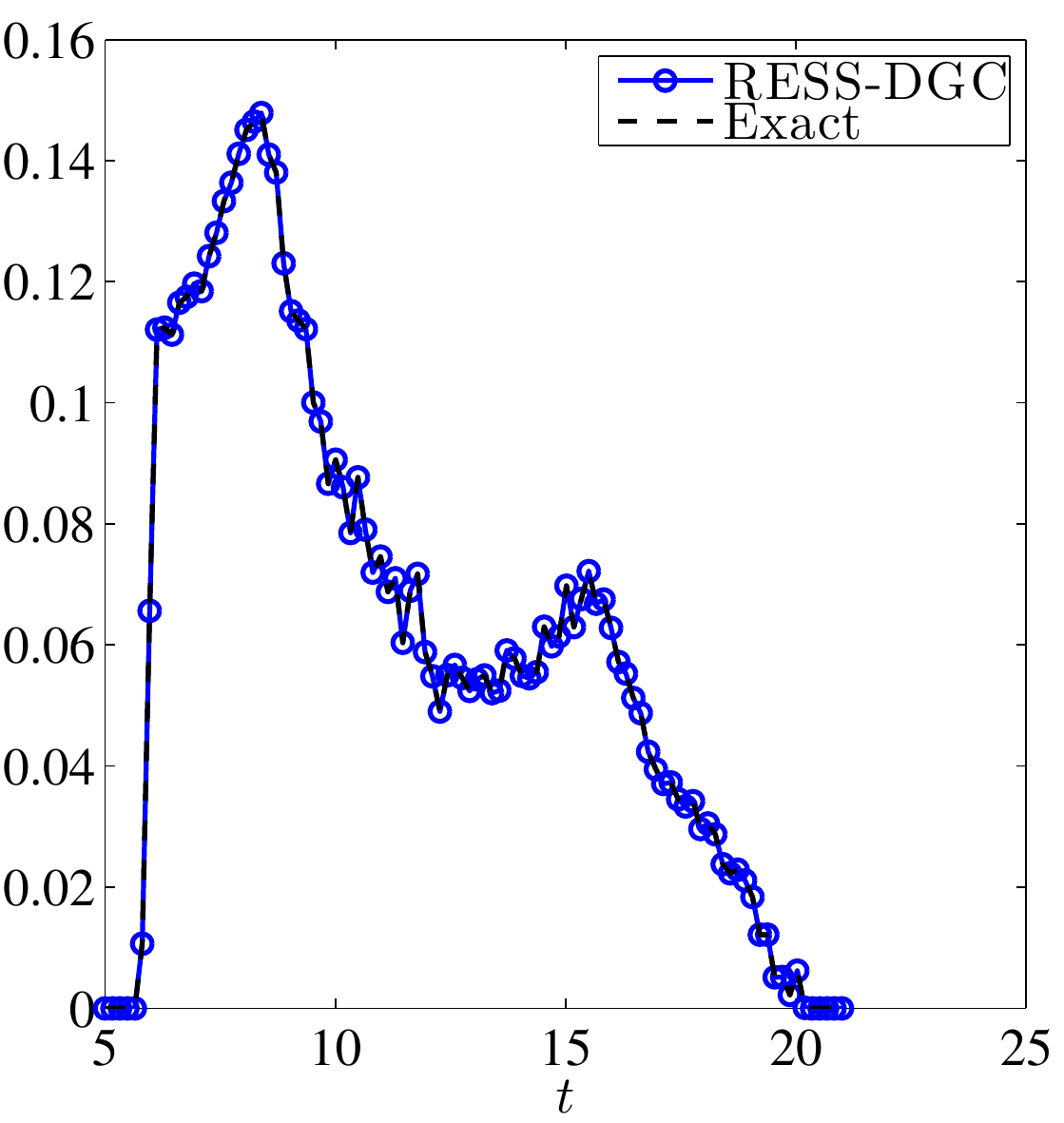}}
    \quad
    \subfloat[(b)]{\includegraphics[width=0.3\textwidth]{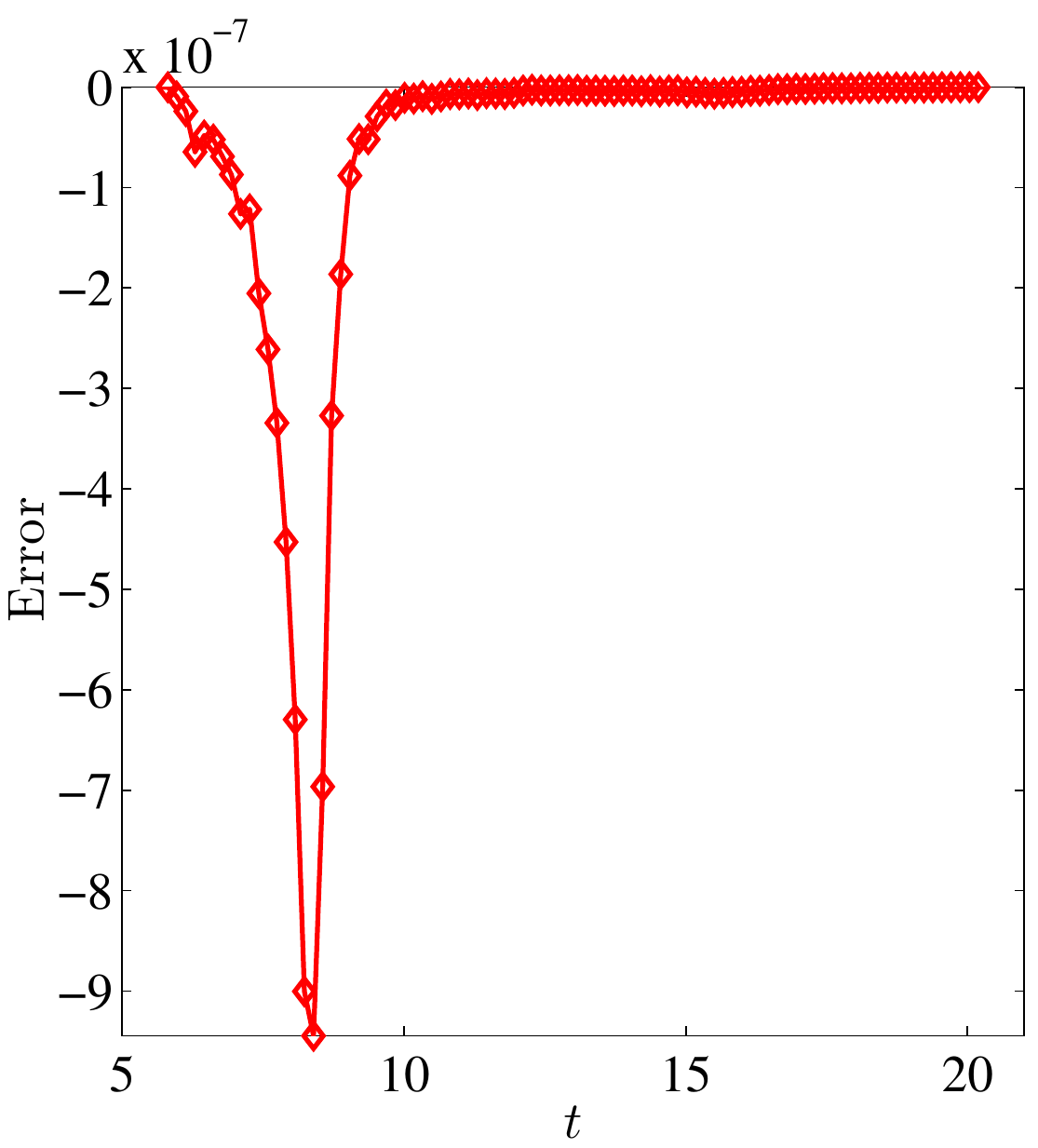}}
  \end{center}
  \caption{\REV{For the shwater matrix, (a) numerically computed DOS by RESS-DGC, compared to
  the exact DOS, and (b) the error of the DOS computed by RESS-DGC.}}
  \label{fig:shwatererror}
\end{figure}

\begin{table}[h]
  \centering
  \REV{
  \begin{tabular}{c|c|c|c}
    \hline
    Matrix & $N_{v}$ & Error RESS-DGC & Error DGC  \\ 
    \hline
    pe3k & 100 & $4.2\times 10^{-2}$ & $1.7\times 10^{-2}$ \\
    pe3k & 200 & $4.0\times 10^{-4}$ & $1.3\times 10^{-2}$ \\
    pe3k & 300 & $4.8\times 10^{-7}$ & $1.1\times 10^{-2}$ \\
    \hline
    shwater & 320 & $9.6\times 10^{-3}$ & $5.7\times 10^{-3}$ \\
    shwater & 480 & $1.2\times 10^{-4}$ & $4.8\times 10^{-3}$ \\
    shwater & 640 & $9.8\times 10^{-7}$ & $3.7\times 10^{-3}$ \\
    \hline
  \end{tabular}
  }
  \caption{Error of the DOS estimation for RESS-DGC and DGC with
  different numbers of random vectors $N_{v}$. }
  \label{tab:errorgeneral}
\end{table}

\begin{table}[h]
  \centering
  \REV{
  \begin{tabular}{c|c|c|c}
    \hline
    Matrix & $N_{v}$ & Time RESS-DGC (s) & Time DGC (s)  \\ 
    \hline
    pe3k & 100 & $780$ & $764$ \\
    pe3k & 200 & $1691$ & $1576$ \\
    pe3k & 300 & $2825$ & $2720$ \\
    \hline
    shwater & 320 & $7371$ & $5513$ \\
    shwater & 480 & $12310$ & $9487$ \\
    shwater & 640 & $23479$ & $18495$\\
    \hline
  \end{tabular}
  }
  \caption{Running time of the DOS estimation for RESS-DGC and DGC with
  different numbers of random vectors $N_{v}$.}
  \label{tab:timegeneral}
\end{table}

\section{Conclusion and future work}\label{sec:conclusion}

For large Hermitian matrices that the only affordable operation is
matrix-vector multiplication, randomized algorithms  can be an
effective way for obtaining a rough estimate the DOS. However, so far 
randomized algorithms are based on Hutchinson's method, which does not use the correlated information among
different random vectors. The accuracy is inherently limited to
$\Or(1/\sqrt{N_{v}})$ where $N_{v}$ is the number of random vectors. 

We demonstrate that randomized low rank decomposition can be used as a
different mechanism to estimate the DOS.  By properly taking into
account the correlated information among the random vectors, we develop
a spectrum sweeping (SS) method that can sweep through the spectrum with
a reasonably small number of random vectors and the accuracy can be
substantially improved compared to $\Or(1/\sqrt{N_{v}})$. However, For
spectrally uniformly distributed matrices with a large number of points
to evaluate the DOS, the direct implementation of the spectrum sweeping
method can have $\Or(N^3)$ complexity. We also present a robust and
efficient implementation of the spectrum sweeping method (RESS).
\REV{For spectrally uniformly distributed matrices}, the
complexity for obtaining the DOS can be improved to $\Or(N^2)$, and the
extra robustness comes from a hybridization with Hutchinson's method for
estimating the residual.  

We demonstrate how the regularized DOS can be used for estimating the
trace of a smooth matrix function. This is based on a careful balance
between the smoothness of the function and that of the DOS.  
Such balance is implemented through a deconvolution procedure. Numerical
results indicate that this allows the accurate estimate of the trace
with again $\Or(N^2)$ scaling.

The current implementation of the spectrum sweeping method is based on
Chebyshev polynomials. Motivated from the discussion
in~\cite{LinSaadYang2015}, Lanczos method would be more efficient than
Chebyshev polynomials for estimating the DOS, and it would be
interesting to extend the idea of spectrum sweeping to Lanczos method
and compare with Chebyshev polynomials.
We also remark that the spectrum sweeping method effectively builds a low
rank decomposition near each point on the spectrum for which the DOS is
to be evaluated. Combining the deconvolution procedure as in the trace
estimation and the low rank decomposition could be potentially useful in
some other applications to directly estimate the whole or part of a
matrix function.  For instance, in electronic structure calculation, the
diagonal entries of the Fermi-Dirac function is needed to evaluate the
electron density. These directions will be explored in the future.

%
%
%

\section*{Acknowledgments}

This work is partially supported by the Scientific Discovery through
Advanced Computing (SciDAC) program and the Center for Applied
Mathematics for Energy Research Applications (CAMERA) funded by U.S.
Department of Energy, Office of Science, Advanced Scientific Computing
Research and Basic Energy Sciences, and by the Alfred P. Sloan
fellowship. 
\REV{We thank the anonymous referees for their comments that greatly
improved the paper.}


\begin{thebibliography}{10}
\expandafter\ifx\csname url\endcsname\relax
  \def\url#1{\texttt{#1}}\fi
\expandafter\ifx\csname urlprefix\endcsname\relax\def\urlprefix{URL }\fi
\expandafter\ifx\csname href\endcsname\relax
  \def\href#1#2{#2} \def\path#1{#1}\fi

\bibitem{Schwartz1966}
L.~Schwartz, Mathematics for the physical sciences, Dover, New York, 1966.

\bibitem{ByronFuller1992}
F.~W. Byron, R.~W. Fuller, Mathematics of classical and quantum physics, Dover,
  New-York, 1992.

\bibitem{RichtmyerBeiglbock1981}
R.~D. Richtmyer, W.~Beiglb{\"o}ck, Principles of advanced mathematical physics,
  Vol.~1, Springer-Verlag New York, 1981.

\bibitem{Ducastelle1970}
F.~Ducastelle, F.~Cyrot-Lackmann, Moments developments and their application to
  the electronic charge distribution of d bands, J. Phys. Chem. Solids 31
  (1970) 1295--1306.

\bibitem{Turek88}
I.~Turek, A maximum-entropy approach to the density of states within the
  recursion method, J. Phys. C 21 (1988) 3251.

\bibitem{DraboldSankey1993}
D.~A. Drabold, O.~F. Sankey, Maximum entropy approach for linear scaling in the
  electronic structure problem, Phys. Rev. Lett. 70 (1993) 3631--3634.

\bibitem{WheelerBlumstein1972}
J.~C. Wheeler, C.~Blumstein, Modified moments for harmonic solids, Phys. Rev. B
  6 (1972) 4380--4382.

\bibitem{SilverRoder1994}
R.~N. Silver, H.~R{\"o}der, Densities of states of mega-dimensional
  {Hamiltonian} matrices, Int. J. Mod. Phys. C 5 (1994) 735--753.

\bibitem{Wang1994}
L.-W. Wang, Calculating the density of states and optical-absorption spectra of
  large quantum systems by the plane-wave moments method, Phys. Rev. B 49
  (1994) 10154.

\bibitem{WeiseWelleinAlvermannEtAl2006}
A.~Wei{\ss}e, G.~Wellein, A.~Alvermann, H.~Fehske, The kernel polynomial
  method, Rev. Mod. Phys. 78 (2006) 275--306.

\bibitem{CovaciPeetersBerciu2010}
L.~Covaci, F.~M. Peeters, M.~Berciu, Efficient numerical approach to
  inhomogeneous superconductivity: The {C}hebyshev-{B}ogoliubov-de {G}ennes
  method, Phys. Rev. Lett. 105 (2010) 167006.

\bibitem{JungCzychollKettemann2012}
D.~Jung, G.~Czycholl, S.~Kettemann, Finite size scaling of the typical density
  of states of disordered systems within the kernel polynomial method, Int. J.
  Mod. Phys. Conf. Ser. 11 (2012) 108.

\bibitem{SeiserPettiforDrautz2013}
B.~Seiser, D.~G. Pettifor, R.~Drautz, Analytic bond-order potential expansion
  of recursion-based methods, Phys. Rev. B 87 (2013) 094105.

\bibitem{HaydockHeineKelly1972}
R.~Haydock, V.~Heine, M.~J. Kelly, Electronic structure based on the local
  atomic environment for tight-binding bands, J. Phys. C: Solid State Phys. 5
  (1972) 2845.

\bibitem{ParkerZhuHuangEtAl1996}
G.~A. Parker, W.~Zhu, Y.~Huang, D.~Hoffman, D.~J. Kouri, Matrix
  pseudo-spectroscopy: iterative calculation of matrix eigenvalues and
  eigenvectors of large matrices using a polynomial expansion of the {Dirac}
  delta function, Comput. Phys. Commun. 96 (1996) 27--35.

\bibitem{GreengardRokhlin1987}
L.~Greengard, V.~Rokhlin, A fast algorithm for particle simulations, J. Comput.
  Phys. 73 (1987) 325--348.

\bibitem{Hackbusch1999}
W.~Hackbusch, A sparse matrix arithmetic based on $\mathcal{H}$-matrices.
  {P}art {I}: {I}ntroduction to $\mathcal{H}$-matrices., Computing 62 (1999)
  89--108.

\bibitem{CandesDemanetYing2009}
E.~Cand{\`e}s, L.~Demanet, L.~Ying, A fast butterfly algorithm for the
  computation of fourier integral operators, SIAM Multiscale Model. Simul.
  7~(4) (2009) 1727--1750.

\bibitem{Hutchinson1989}
M.~F. Hutchinson, A stochastic estimator of the trace of the influence matrix
  for {Laplacian} smoothing splines, Commun. Stat. Simul. Comput. 18 (1989)
  1059--1076.

\bibitem{LinSaadYang2015}
L.~Lin, Y.~Saad, C.~Yang, Approximating spectral densities of large matrices,
  SIAM Rev. in press.

\bibitem{AvronToledo2011}
H.~Avron, S.~Toledo, Randomized algorithms for estimating the trace of an
  implicit symmetric positive semi-definite matrix, J. ACM 58 (2011) 8.

\bibitem{Parlett1980}
B.~N. Parlett, The symmetric eigenvalue problem, Vol.~7, SIAM, 1980.

\bibitem{SakuraiSugiura2003}
T.~Sakurai, H.~Sugiura, A projection method for generalized eigenvalue
  problems, J. Comput. Appl. Math. 159 (2003) 119--128.

\bibitem{Polizzi2009}
E.~Polizzi, Density-matrix-based algorithm for solving eigenvalue problems,
  Phys. Rev. B 79 (2009) 115112--115117.

\bibitem{SchofieldChelikowskySaad2012}
G.~Schofield, J.~R. Chelikowsky, Y.~Saad, A spectrum slicing method for the
  {K}ohn-{S}ham problem, Comp. Phys. Comm. 183 (2012) 497--505.

\bibitem{FangSaad2012}
H.-R. Fang, Y.~Saad, A filtered {L}anczos procedure for extreme and interior
  eigenvalue problems, SIAM J. Sci. Comput. 34 (2012) A2220--A2246.

\bibitem{AktulgaLinHaineEtAl2014}
H.~M. Aktulga, L.~Lin, C.~Haine, E.~G. Ng, C.~Yang, Parallel eigenvalue
  calculation based on multiple shift--invert {L}anczos and contour integral
  based spectral projection method, Parallel Comput. 40 (2014) 195--212.

\bibitem{Demko1977}
S.~Demko, Inverses of band matrices and local convergence of spline
  projections, SIAM J. Numer. Anal. 14~(4) (1977) 616--619.

\bibitem{BenziBoitoRazouk2013}
M.~Benzi, P.~Boito, N.~Razouk, Decay properties of spectral projectors with
  applications to electronic structure, SIAM Rev. 55~(1) (2013) 3--64.

\bibitem{Lin2015}
L.~Lin, Localized spectrum slicing, arXiv: 1411.6152.

\bibitem{Meinardus1967}
G.~Meinardus, Approximation of functions: Theory and numerical methods,
  Springer, 1967.

\bibitem{GolubVan2013}
G.~H. Golub, C.~F. Van~Loan, Matrix computations, 4th Edition, Johns Hopkins
  Univ. Press, Baltimore, 2013.

\bibitem{LibertyWoolfeMartinssonEtAl2007}
E.~Liberty, F.~Woolfe, P.~Martinsson, V.~Rokhlin, M.~Tygert, Randomized
  algorithms for the low-rank approximation of matrices, Proc. Natl. Acad. Sci.
  USA 104 (2007) 20167--20172.

\bibitem{WoolfeLibertyRokhlinEtAl2008}
F.~Woolfe, E.~Liberty, V.~Rokhlin, M.~Tygert, A fast randomized algorithm for
  the approximation of matrices, Appl. Comput. Harmon. Anal. 25~(3) (2008)
  335--366.

\bibitem{HalkoMartinssonTropp2011}
N.~Halko, P.-G. Martinsson, J.~A. Tropp, Finding structure with randomness:
  Probabilistic algorithms for constructing approximate matrix decompositions,
  SIAM Rev. 53~(2) (2011) 217--288.

\bibitem{FloridaMatrix}
T.~A. Davis, Y.~Hu, {The University of Florida sparse matrix collection}, ACM
  Trans. Math. Software 38 (2011) 1.

\bibitem{Knyazev2001}
A.~V. Knyazev, Toward the optimal preconditioned eigensolver: Locally optimal
  block preconditioned conjugate gradient method, SIAM J. Sci. Comp. 23 (2001)
  517--541.

\end{thebibliography}

\end{document}